\documentclass[11pt]{article}
\usepackage{psfrag}
\usepackage{graphicx}
\usepackage{latexsym}
\usepackage{amssymb}
\usepackage{amsfonts}
\usepackage{amsmath}
\usepackage{amsthm}

\newtheorem{theorem}{Theorem}[section]
\newtheorem{corollary}[theorem]{Corollary}
\newtheorem{lemma}[theorem]{Lemma}
\newtheorem{proposition}[theorem]{Proposition}

\theoremstyle{remark}
\newtheorem{remark}[theorem]{Remark}

\theoremstyle{remark}
\newtheorem{definition}[theorem]{Definition}

\theoremstyle{remark}

\numberwithin{equation}{section}

 \newcommand{\ep}{\epsilon}
\newcommand{\del}{\delta}

\newcommand{\reals}{\mathbb{R}}
\newcommand{\ints}{\mathbb{Z}}
\newcommand{\nats}{\mathbb{N}}

\newcommand{\disk}{{\mathbb{D}^2}}

\newcommand{\sphere}{{\mathbb{S}}}

\newcommand{\nbhd}{\mathcal{N}}

\newcommand{\inv}{^{-1}}

\newcommand{\ovl}{\bar}

\newcommand{\til}{\widetilde}

\newcommand{\Hdim}{\mathcal{H}}

\newcommand{\subeq}{\subseteq}
\newcommand{\supeq}{\supseteq}
\newcommand{\mand}{\quad \hbox{and} \quad}
\newcommand{\mbf}{\mathbf}
\newcommand{\adm}{\mathcal{A}}
\newcommand{\bslash}{\backslash}
\newcommand{\ul}{\underline}

\def\diam{\operatorname{diam}}

\def\card{\operatorname{card}}
\def\dist{\operatorname{dist}}
\def\im{\operatorname{im}}

\def\lng{\operatorname{length}}
\def\ins{\operatorname{inside}}
\def\modu{\operatorname{mod}}
\def\cl{\operatorname{cl}}
\def\ind{\operatorname{ind}}

\begin{document}
\title{Quasisymmetric structures on surfaces \footnote{2000 MSC Primary 30C65, Secondary 57M50}}
\author{Kevin Wildrick \thanks{The author was partially supported by NSF grants DMS 0244421, DMS 0456940, and DMS-0602191.}\\
University of Michigan Department of Mathematics\\
Ann Arbor, MI 48019\\ 
 kwildric@umich.edu \\ }
\date{}
\maketitle

\abstract{We show that a locally Ahlfors $2$-regular and locally linearly locally contractible metric surface is locally quasisymmetrically equivalent to the disk.  We also discuss an application of this result to the problem of characterizing surfaces in some Euclidean space that are locally bi-Lipschitz equivalent to an open subset of the plane.}

\section{Introduction}
\indent

Quasisymmetric mappings are a generalization of conformal mappings to metric spaces.  Recent work \cite{B-K}, \cite{me} has provided a substantial existence theory, analogous to the classical Uniformization Theorem for Riemann surfaces, for quasisymmetric mappings on metric spaces that satisfy simple geometric conditions.  Quasisymmetric uniformization results have been applied to diverse subjects, including Cannon's conjecture regarding Gromov hyperbolic groups with two-sphere boundary \cite{B-K} and the quasiconformal Jacobian conjecture \cite{QCJ}.  

In this paper, we develop a local existence theory for quasisymmetric mappings on general surfaces.  In other words, we give simple geometric conditions for a metric space homeomorphic to a surface to be considered a generalized Riemann surface.  The geometric conditions we consider are localized versions of Ahlfors regularity and linear local contractibility.  We defer the precise definitions to Section \ref{defs}.

The global versions of these conditions have been studied extensively \cite{QuantTop}, \cite{B-K}, \cite{me}.  A deep theorem of Semmes \cite[Theorem B.10]{QuantTop} states that if a metric space homeomorphic to a connect, orientable, $n$-manifold is complete, Ahlfors $n$-regular and linearly locally contractible, then it supports a weak $(1,1)$-Poincar\'e inequality.  Spaces which support such an inequality can be considered to have ``good calculus", and they are the preferred environment for the theory of quasiconformal mappings on metric spaces \cite{Acta}.  Moreover, they enjoy several geometric properties, such as quasiconvexity \cite{Korte}. 

Our main result is the following theorem.

\begin{theorem}\label{main II} Let $(X,d)$ be a proper, locally Ahlfors $2$-regular, and locally linearly locally contractible ($LLLC$) metric space homeomorphic to a surface.  Then each point of $X$ has a neighborhood that is quasisymmetrically equivalent to the disk.  \end{theorem}

Our proof is in fact quantitative, and provides good bounds on the size of the resulting quasidisk.  See Theorem \ref{existence} for the complete result.  An outline of the proof of Theorem \ref{main II} is as follows.  The main task is to construct an Ahlfors $2$-regular and linearly locally connected (a slightly weaker version of linear local contractibility) planar neighborhood of a given point $z \in X$.  We then apply previously established uniformization results from \cite{me} to produce the desired quasisymmetric mapping.  The obstacle is that compact subsets of a locally Ahlfors $2$-regular and $LLLC$ metric space need not be Ahlfors $2$-regular and linearly locally connected.  However, we show that if $\gamma$ is a quasicircle contained in a planar subset of $X$, then the closed Jordan domain defined by $\gamma$ is Ahlfors $2$-regular and linearly locally connected.  Thus it suffices to construct a quasicircle at a specificed scaled that surrounds a given point $p \in X$.  

In constructing the quasicircle, we first show that $(X,d)$ is locally quasiconvex.  This would follow from the result of Semmes \cite[Theorem B.10]{QuantTop}, except that we consider localized conditions.  As it is, our methods resemble those employed by Semmes.  As we have specialized to two dimensions, our proof is fairly elementary and direct.  We also indicate how our proof could be upgraded to give a full local analogue of Semmes result in two dimensions. 

A locally compact and locally quasiconvex space is, up to a locally bi-Lipschitz change of metric, locally geodesic.  With this simplification, we employ discrete methods to construct a loop surrounding $z$ of controlled length and lying in a controlled annulus.  We then solve an extremal problem to produce the desired quasicircle. 

The author would like to thank Mario Bonk and Juha Heinonen for many contributions to this paper.  Also, thanks to Urs Lang for a suggestion which substantially improved the proof of Theorem \ref{quasicircle}.   

\section{Application to local bi-Lipschitz parameterizations}\label{app}
\indent

Theorem \ref{main II} plays a role in the program of Heinonen and others to give necessary conditions for a submanifold of some $\reals^N$ to be locally bi-Lipschitz equivalent to $\reals^2$.  We now give a brief description of this program.  See \cite{JuhaICM}, \cite{JuhaRick}, and \cite{JuhaSull} for a full exposition. 

An oriented, $n$-dimensional submanifold $X$ of $\reals^N$ admits local Cartan-Whitney presentations if for each point $p \in X$ there is an $n$-tuple of flat $1$-forms $\rho = (\rho_1, \hdots, \rho_n)$ defined on an $\reals^N$ neighborhood of $p$, such that near $p$ on $X$, there is a constant $c > 0$ such that 
$$\ast(\rho_1 \wedge \hdots \wedge \rho_n) \geq c > 0.$$
Here $\ast$ denotes the Hodge star operator.  

\begin{theorem}[Heinonen-Sullivan \cite{JuhaSull}]\label{presentation} Let $X \subeq \reals^N$ be an $n$-manifold endowed with the metric inherited from $\reals^N$.  If $X$ is locally Ahlfors $n$-regular,  $LLLC$, and admits local Cartan-Whitney presentations, then $X$ is locally bi-Lipschitz equivalent to $\reals^n$, except on a closed set of measure $0$ and topological dimension at most $n-2$.  
\end{theorem} 

We now give the basic outline of the proof of Theorem \ref{presentation}.  Given $p \in X$, let $\rho$ be a Cartan-Whitney presentation near $p \in X$.  It is shown that on a sufficiently small open neighborhood $U$ of $p$, the map $f\colon U \to \reals^n$ defined by 
\begin{equation}\label{qr} f(x) = \int_{[p,x]} \rho \end{equation}
is discrete, open, and sense preserving with volume derivative uniformly bounded away from $0$.  If follows that $f$ is a branched covering that is locally bi-Lipschitz off its branch set, which is of measure $0$ and topological dimension at most $n-2$.  

If $n=2$, then this branch set consist of isolated points.  Bonk and Heinonen \cite{JuhaICM} noted that in this case, the measurable Riemann mapping theorem and Theorem \ref{main II} provide a resolution of this branch set, proving the following theorem:

\begin{theorem}[Bonk-Heinonen \cite{JuhaICM}]\label{application}  Let $X \subeq \reals^N$ be a surface endowed with the metric inherited from $\reals^N$.  If $X$ is locally Ahlfors $2$-regular, $LLLC$, and admits local Cartan-Whitney presentations, then $X$ is locally bi-Lipschitz equivalent to a subset of $\reals^2$.  
\end{theorem}
 
As no proof of Theorem \ref{application} is available in the literature, we now sketch a proof of Theorem \ref{application}, as communicated by Bonk and Heinonen.  Using the quantitative result Theorem \ref{existence}, this proof could be made quantitative as well.  

\begin{proof}  Let $p \in X$.  As above, we may find an open neighborhood $U$ of $p$ such that the map $f$ given by \eqref{qr} is well defined.  As $n=2$, we may assume that no point of $U \slash \{p\}$ is a branch point of $f$.  If $p$ is not a branch point of $f$, we may find a smaller open neighborhood $U'$ of $p$ on which $f$ is bi-Lipschitz, as desired. 

Thus we may assume that $p$ is a branch point of $f$.  Then the local degree of $f$ at $p$ is an integer $k \geq 2$. By Theorem \ref{main II}, there is an open neighborhood $U' \subeq U$ of $p$ and a quasisymmetric map $\phi \colon U' \to \disk$.  Now $h:=f \circ \phi \inv$ is a quasiregular mapping from $\disk$ to an open set $\Omega \subeq \reals^2$.  The mapping $h$ defines a Beltrami differential $\mu_h$ on $\disk$.  By the measurable Riemann mapping theorem, there is a quasiconformal mapping $g \colon \disk \to \Omega$ such that the Beltrami differential $\mu_g$ agrees with $\mu_h$ almost everywhere.  The uniqueness statement of the measurable Riemann mapping theorem allows us to assume that $\Omega = \disk$ and $h(z) = g(z)^k$.  Let $\rho \colon \disk \to \disk$ be the radial stretching map
$$\rho(z) = |z|^{k-1}z,$$
and define a homeomorphism $\psi \colon U' \to \disk$ by $\psi = \rho \circ g \circ \phi.$
Then we have the following relationship amongst Jacobian determinants:
$$J_f = J_{h}J_{\phi}  = k J_{\rho \circ g} J_{\phi} = k J_{\psi}.$$
Hence the volume derivative of  $\psi$ is bounded above and below by a constant multiple of the volume derivative of $f$.  Since $\psi$ is quasiconformal, this implies that $\psi$ is bi-Lipschitz, completing the proof. 
\end{proof}

\section{Notations, definitions and preliminary results}\label{defs}

\subsection{Metric spaces}\label{metric spaces}
\indent 

We will often denote a metric space $(X,d)$ by $X$.  Given a point $x \in X$ and a number $r>0$, we define the open and closed balls centered at $x$ of radius $r$ by
$$B_{(X,d)}(x, r) = \{y \in X: d(x,y) < r\} \mand \ovl{B}_{(X,d)}(x,r)= \{y \in X: d(x,y) \leq r\}.$$
Where it will not cause confusion, we denote $B_{(X,d)}(x,r)$ by $B_{X}(x,r)$, $B_d(x,r)$, or $B(x,r)$.  A similar convention is used for all other notions which depend implicitly on the underlying metric space. The diameter of a subset $E$ of $(X,d)$ is denoted by 
$\diam(E),$ and the distance between two subsets $E, F \subeq X$ is denoted by $\dist(E,F).$ For $\ep > 0$, the $\ep$-neighborhood of $E \subeq X$ is given by
$$\nbhd_{\ep}(E) = \bigcup_{x \in E} B(x, \ep).$$
We denote the completion of a metric space $X$ by $\ovl{X}$, and define the metric boundary of $X$ by $\partial{X} =\ovl{X}\bslash X$.  A metric space is said to be proper if every closed ball is compact. 

Let $[a,b]\subset \reals$ be a compact interval.  A continuous map $\gamma \colon [a,b] \to X$ is called a path in $X$.  The path $\gamma$ may also be referred to as a parameterization of its image $\im{\gamma}$.  If $\gamma$ happens to be an embedding, then $\im{\gamma} = \gamma([a,b])$ is called an arc in $X$.   If $\alpha$ is an arc in $X$, and $u, v\in \alpha$, then the segment $\alpha[u,v] \subeq X$ is well defined. 

A path $\gamma \colon [a,b] \to X$ is rectifiable if $\lng(\gamma) < \infty$.  The metric space $(X,d)$ is $L$-quasiconvex, $c \geq 1$, if every pair of points $x,y \in X$ may be joined by a path in $X$ of length no more than $Ld(x,y)$.  We say that $(X,d)$ is locally quasiconvex if for every compact subset $K$ of $X$, there is a constant $L_K$ such that every pair of points $x, y \in K$ may be joined by a path in $X$ of length no more than $L_Kd(x,y)$.  

Given a rectifiable path $\gamma$ of length $L$, and a Borel function $\rho \colon X \to [0,\infty]$, we define the path integral
$$\int_{\gamma} \rho \ ds = \int_{0}^{L} (\rho \circ \gamma_l)(t)\ dt,$$
where $\gamma_l$ denotes the arclength reparameterization of $\gamma$.  

If $\gamma$ is a path connecting points $x,y \in X$ and satisfying 
$$d(x,y) =\lng(\gamma),$$
then $\gamma$ is said to be a geodesic path.  If points $x,y\in X$ may be connected by a geodesic, we define the geodesic segment $[x,y]$ to be image of some geodesic with endpoints $x$ and $y$.  By convention, we assume that $[x,y]=[y,x]$ as sets, and that if $w \in [x,y]$, then $[x,w]$ and $[w,y]$ are subsets of $[x,y]$. Given a geodesic segment $[x,y]$, we denote by $s^{[x,y]}\colon [0, d(x,y)] \to X$ the arc length parameterization of $[x,y]$, with the convention that $s^{[x,y]}$ has initial value $x$ and terminal value $y$.  If $[x,y]$ is a geodesic segment with $x \neq y$, given any continuum $[a,b] \subeq \reals$, we define the standard parameterization of $[x,y]$ by $[a,b]$ to be given by $s^{[x,y]}_{[a,b]}\colon [a,b] \to X$ where
$$s^{[x,y]}_{[a,b]}(t) = s^{[x,y]}\left(\frac{t-a}{b-a}d(x,y)\right).$$

Metric spaces admit arbitrarily fine approximations by discrete spaces in the following sense.  Given $\ep > 0$, a subset $S \subeq X$ is $\ep$-separated if $d(a,b) \geq \ep$ for all pairs of distinct points $a,b \in S$.  By Zorn's lemma, maximal $\ep$-separated sets exist for every $\ep > 0$, and for such sets the collection $\{B(a, \ep)\}_{a \in S}$ covers $X$.  

We denote by $\sphere^2, \reals^2,$ and $\disk$ the sphere, the plane, and the disk, each equipped with the standard metric inherited from the ambient Euclidean space.  

For specificity, we define $\sphere^1:=[0, 2\pi)$ as a set, and topologized and metrized as a subset of the plane under the identification $\theta \leftrightarrow e^{i\theta}.$  A continuous map of $\sphere^1$ to a space $X$ is called a loop in $X$. We define length and integrals for loops as for paths, with obvious modifications.  A collection of points $\{\theta_1,\hdots, \theta_n\}\subeq \sphere^1$ is said to be in cyclic order if they are ordered according to the standard positive orientation on $\sphere^1$.  Given a cyclically ordered collection of points $\{\theta_1, \hdots, \theta_n\}$ containing at least three distinct points, we may unambiguously define the arcs $[\theta_i, \theta_{i+1}]$, $i = 1, \hdots, n, \modu n$, that lie between consecutive points.  

\subsection{Finite dimensional metric spaces}\label{dimension}
\indent 

A metric space $(X,d)$ is said to be doubling if there exists a non-negative integer $N$ such that for each $a \in X$ and $r > 0$, the ball $B(a, r)$ may be covered by at most $N$ balls of radius $r/2$.  If $(X,d)$ is doubling, then we may find constants $Q \geq 0$ and $C \geq 1$, depending only on $N$, such that for all $0<\ep \leq 1/2$, each ball $B(a, r)$ may be covered by at most $C\ep^{-Q}$ balls of radius $\ep r$.  The infimum over such $Q$ is called the Assouad dimension of $(X,d)$.  

For $Q \geq 0$, we will denote the $Q$-dimensional Hausdorff measure on $(X,d)$ by $\Hdim^Q_{(X,d)}$.  For a full description of Hausdorff measure, see \cite[2.10]{Federer}. 

\begin{definition}\label{basic reg def}A metric space $(X,d)$ is called Ahlfors $Q$-regular, $Q\geq 0$, if there exists a constant $C\geq 1$ such that for all $a \in X$ and $0<r\leq \diam{X}$, we have
\begin{equation}\label{strong reg def} \frac{r^Q}{C} \leq \Hdim^Q(\ovl{B}_d(a,r)) \leq Cr^Q.\end{equation}
Note that if the upper bound in \eqref{strong reg def} is valid for all $0< r \leq \diam(X)$, then it is also valid for all $r > \diam(X)$ as well.  
\end{definition}

An Ahlfors $Q$-regular metric space can be thought of as $Q$-dimensional at every scale. For example, the space $\reals^2$ is Ahlfors $2$-regular, while the infinite strip
$$\{(x,y) \in \reals^2 : 0 < y < 1\}$$
is not Ahlfors $Q$-regular for any $Q$.  At small scales, the strip appears two-dimensional, while at large scales it appears one-dimensional.  

For an in-depth discussion of Ahlfors regularity, see \cite[Appendix C]{QuantTop}.

\begin{definition}\label{loc reg definition} A metric space $(X,d)$ is locally Ahlfors $Q$-regular, $Q\geq 0$, if for every compact set $K \subeq X$ there exists a constant $C_K \geq 1$ and a radius $R_K >0$ such that for all $x \in K$ and $0<r\leq R_K$, we have
\begin{equation}\label{loc reg def} {C_K}\inv r^Q \leq \Hdim^Q(\ovl{B}_d(x,r)) \leq C_Kr^Q.\end{equation}
\end{definition}

Note that this definition is localized in two ways: the constant $C_K$ depends on the location of the center of the ball under consideration, and the radius $R_K$ restricts the scales to which the condition applies at this location.   We will only apply this definition to spaces which have many compact subsets, i.e., proper spaces.  

It will be convenient to have a notion where the radius $R_K$ is tied to the size of the set under consideration.

\begin{definition}\label{rel reg definition} A subset of $U$ of a metric space $(X,d)$ is called relatively Ahlfors $Q$-regular, if there exists a constant $C \geq 1$ such that for all $0 < r \leq \diam(U)$ and all $x \in U$, 
\begin{equation}\label{rel reg ineq} {C}\inv r^Q \leq \Hdim^Q(\ovl{B}_X(x,r)) \leq Cr^Q.\end{equation}
\end{definition}

Note that in the definition of a relatively Ahlfors regular set $U$, the balls under consideration may contain points outside of $U$; we require \eqref{rel reg ineq} to hold for $B_X(x, r)$ and not $B_U(x,r)$.  Hence a relatively Ahlfors regular set need not be Ahlfors regular as a metric space.

To state some of our theorems in full generality, we also employ a relative doubling condition.  

\begin{definition}The relative Assouad dimension of a subset $U$ of a metric space $(X,d)$ is the infimum of all $Q \geq 0$ such that there exists a constant $D \geq 1$ with the property that for all $0 < r \leq \diam(U)$, all $x \in U$, and all $0 < \ep \leq 1/2$, the ball $B_X(x, r)$ can be covered by at most $D\ep^{-Q}$ balls in $X$ of radius $\ep r$.
\end{definition}

We now give a local version of the fact that Ahlfors regular spaces are doubling. 

\begin{proposition}\label{loc regular implies loc doubling} Let $(X,d)$ be a locally Ahlfors $Q$-regular metric space, and let $K \subeq X$ be compact.  Let $R_K$ and $C_K$ be the constants associated to $K$ by the local Ahlfors $Q$-regularity condition.  If $U$ is a subset of $X$ such that the $2\diam(U)$-neighborhood of $U$ is contained in $K$ and $\diam(U) \leq R_K/2$, then $U$ has relative Assouad dimension $Q$ with constant $C_K^28^Q.$ \end{proposition}

\begin{proof} Let $x \in U$, $0 < r \leq \diam U,$ and $0 < \ep \leq 1/2.$ Then  $x \in K$,  and $0< r \leq R_K/2$.  Let $\{x_i\}_{i \in I}$ be a maximal $\ep r$-separated set in $B(x,r)$.   Then $\{B(x_i,\ep r)\}_{i \in I}$ covers $B(x,r)$, while $\{\ovl{B}(x_i,\ep r/4)\}_{i \in I}$ is disjointed.  Since $\ep < 1/2$, we see that for all $i \in I$, $B(x_i, \ep r) \subeq \ovl{B}(x, 2r)$.  Since the $2\diam(U)$-neighborhood of $U$ is contained in $K$, we see that $x_i \in K$.  We may now apply the local Ahlfors $Q$-regularity condition to see that 
$$\frac{\card(I)}{C_K}\left( \frac{\ep r}{4}\right)^Q \leq \sum_{i \in I} \Hdim^Q\left(\ovl{B}(x_i, \ep r/4)\right) \leq \Hdim^Q(\ovl{B}(x, 2r)) \leq C_K(2r)^Q.$$
This implies that 
$$\card(I) \leq C_K^2 8^Q\ep^{-Q},$$
showing that the relative Assouad dimension of $U$ is at most $Q$ and giving the desired constant. The fact that the relative Assouad dimension of $U$ is precisely $Q$ is similarly straight-forward; since it will not actually be needed later, we leave the proof to the reader. \end{proof}

\subsection{Contractibility and connectivity conditions}\label{LLC}
\indent 

Here we discuss various types of quantitative local connectivity and contractibility.  Perhaps the most basic is the following.

\begin{definition}\label{bounded turning}A subset $E$ of a metric space $(X,d)$ is of $\lambda$-bounded turning in $X$, $\lambda \geq 1$, if each pair of distinct points $x,y \in E$ may be connected by a continuum $\gamma \subeq X$ such that $\diam(\gamma) \leq \lambda d(x,y).$ If $E$ is bounded turning in itself, then it is said to be of bounded turning. 
\end{definition}

Recall that a continuum is a compact connected set containing at least two points.  The condition for a subset $E$ to be of bounded turning in a metric space $(X,d)$ is non-standard; usually only spaces which are of bounded turning in themselves are considered.  The bounded turning condition along with a similar dual condition constitute linear local connectivity. 

\begin{definition}\label{LLC def} Let $\lambda \geq 1$.  A metric space $(X,d)$ is $\lambda$-linearly locally connected ($\lambda$-$LLC$) if for all $a \in X$ and  $r >0$ the following two conditions are satisfied:
\begin{itemize}
\item[(i)]  for each pair of distinct points $x,y \in B(a,r)$, there is a continuum $E \subeq B(a,\lambda r)$ such that $x,y\in E$,  
\item[(ii)] for each pair of distinct points $x,y \in X-B(a,r)$, there is a continuum $E \subeq X-B(a, r/\lambda)$ such that $x,y \in E$. 
\end{itemize}
\end{definition}

Individually, conditions $(i)$ and $(ii)$ are referred to as the $\lambda$-$LLC_1$ and $\lambda$-$LLC_2$ conditions.  Roughly speaking, the $LLC$ condition rules out cusps and bubbles from the geometry of a metric space.  

\begin{remark}\label{LLC term} A space which satisfies the $\lambda$-$LLC_1$ condition is $4\lambda$-bounded turning, and a space which is $\lambda$-bounded turning satisfies  $2\lambda$-$LLC_1$.   The terminology ``linearly locally connected'' stems from the following fact. Let $(X,d)$ satisfy the $\lambda$-$LLC_1$ condition, and let $x \in X$ and $r>0$.  If $C(x)$ be the connected component of $B(x,r)$ containing $x$.  Then $B(x,r/\lambda)\subeq C(x) \subeq B(x,r)$.  
\end{remark}

\begin{definition}\label{global LLcont} A metric space is $\Lambda$-linearly locally contractible, $\Lambda \geq 1$, if for all $a \in X$ and $r \leq \diam(X)/\Lambda$, the ball $B(a, r)$ is contractible inside the ball $B(a, \Lambda r)$.   
\end{definition}

Unfortunately, the term ``linearly locally contractible" has not yet stabilized in the literature.  Our definition is global in nature and agrees with the definitions given in \cite{B-K} and \cite{QuantTop}. The definition given in \cite{JuhaICM} is localized, and agrees with the following:

\begin{definition}\label{LLcont}A metric space $(X,d)$ is locally linearly locally contractible ($LLLC$) if for every compact subset $K \subeq X$, there is a constant $\Lambda_K \geq 1$ and radius $R_K > 0$ such that for every point $x \in K$ and radius $0<r\leq R_K$, the ball $B(x,r)$ is contractible inside the ball $B(x,\Lambda_Kr).$  
\end{definition}

As with Ahlfors regularity, it will be convenient to have a relative version as well.  

\begin{definition}\label{rel LLcont} A subset $U$ of a metric space $(X,d)$ is relatively $\Lambda$-locally linearly contractible if for all $x \in U$ and $0 < r \leq \diam(U)$, the ball $B_X(x, r)$ is contractible inside the ball $B_X(x, \Lambda r).$  
\end{definition}

In certain situations, the $LLC$ and linear local contractibility conditions are equivalent \cite[Lemma 2.5]{B-K}.  We now localize this statement to show that the $LLLC$ condition implies a relative $LLC$ condition for certain sets, quantitatively.  

\begin{definition}\label{rel LLC} An subset $U$ of a metric space $(X,d)$ is relatively $\lambda$-$LLC$, $\lambda \geq 1$, if for all points $x \in U$ and $0 < r \leq \diam(U)$ the following conditions hold:
\begin{itemize}
\item[(i)]  for each pair of distinct points $y,z \in B_X(x,r)$, there is a continuum $\gamma \subeq B_X(x,\lambda r)$ such that $y,z \in \gamma$,  
\item[(ii)] if $B(x,r)$ is compactly contained in $U$, then for each pair of distinct points $y,z \in U\bslash B_X(x,r)$, there is a continuum $\gamma \subeq U\bslash B_X(x, r/\lambda)$ such that $y,z \in \gamma$. 
\end{itemize} 
\end{definition}

\begin{proposition}\label{LLLC implication} Suppose that $(X,d)$ is a linearly locally contractible metric space homeomorphic to a connected topological $n$-manifold, $n \geq 2$, and let $K \subeq X$ be compact.  Let $R_K$ and $\Lambda_K$ be the constants associated with $K$ by the linear local contractibility condition.  If $U \subeq K$ and $\diam(U) \leq R_K$, then $U$ satisfies the first relative $LLC$ condition with constant $\Lambda_K$.  If in addition, $U$ is connected and open in $X$, then $U$ is relatively $\lambda$-$LLC$ for any $\lambda > \Lambda_K.$ \end{proposition}

As the proof of Proposition \ref{LLLC implication} is nearly identical to that of Lemma 2.4 in \cite{B-K}, we omit it. 
 
\begin{remark}\label{arcs} We may consider alternate versions of bounded turning, $LLC$, and relative $LLC$ where continua are replaced with arcs.  These conditions are quantitatively equivalent in locally path connected spaces, as follows.  It follows from \cite[Theorems 3.15 and  3.30]{Hocking} that if $\gamma:[a,b] \to X$ is a path, then there is an arc $\alpha$ in $X$  that connects $\gamma(a)$ and $\gamma(b)$ and is contained in $\im{\gamma}$.  Thus, a simple covering argument shows that if $(X,d)$ is locally path connected, and $E \subeq X$ is a continuum that is contained in an open set $V \subeq X$, then any pair of points $x,y \in E$ is contained in an arc in $V$. 
\end{remark}

\subsection{Quasisymmetric mappings}\label{qs} 
\indent

Quasisymmetric mappings first arose as the restrictions of quasiconformal mappings to the real line \cite{Beurling}. For the basic theory and applications of quasisymmetric mappings, see \cite{QS} and \cite[Chapter 10]{LAMS}.

\begin{definition}\label{qs def}A homeomorphism $f\colon X \to Y$ of metric spaces is called quasisymmetric if there exists a homeomorphism $\eta\colon [0,\infty) \to [0,\infty)$ such that for all triples $a,b,c \in X$ of distinct points, 
$$\frac{d_Y(f(a),f(b))}{d_{Y}(f(a),f(c))} \leq \eta\left(\frac{d_X(a,b)}{d_X(a,c)}\right).$$
\end{definition}

We will call the function $\eta$ the distortion function of $f$; when $\eta$ needs to be emphasized, we say that $f$ is $\eta$-quasisymmetric.  If $f$ is $\eta$-quasisymmetric, then $f\inv$ is also quasisymmetric with distortion function $(\eta\inv(t\inv))\inv$.  Thus we say that metric spaces $X$ and $Y$ are quasisymmetric or quasisymmetrically equivalent if there is a quasisymmetric homeomorphism from $X$ to $Y$. 

The following result, due to V\"ais\"al\"a \cite[Theorems 3.2, 3.10, 4.4, and 4.5]{QM}, shows that the $LLC$ condition is a quasisymmetric invariant.

\begin{theorem}[V\"ais\"al\"a]\label{LLC is qs invariant} If $X$ is a $\lambda$-$LLC$ metric space and $f\colon X \to Y$ is $\eta$-quasisymmetric, then $Y$ is $\lambda'$-$LLC$ for some $\lambda'$ depending only on $\lambda$ and $\eta$. 
\end{theorem}

A metric space that is quasisymmetrically equivalent to $\sphere^1$ is called a quasicircle. Tukia and V\"ais\"al\"a gave the following characterization of quasicircles \cite{QS}. 

\begin{theorem}[Tukia-V\"ais\"al\"a]\label{metric quasicircles} A metric space $(X,d)$ that is homeomorphic to $\sphere^1$ is a quasicircle if and only if $(X,d)$ is doubling and $LLC$.  The doubling and $LLC$ constants can be chosen to depend only on the distortion function of the quasiymmetry, and vice-versa. 
\end{theorem}

\begin{remark}\label{3 point} It is an informative exercise to show that for any Jordan curve $J$ in a metric space $(X,d)$, the $LLC$ condition may be restated in the following more intuitive fashion. Given any two distinct points $x,y\in J$, the set $J\bslash \{x,y\}$ consists of two disjoint arcs $J_1$ and $J_2$.  The Jordan curve $J$ is $LLC$ if and only if there some $\lambda'$ such that for all pairs of distinct points $x, y \in J$,
\begin{equation}\label{three point} \min\{\diam{J_1}, \diam{J_2}\} \leq \lambda' d(x,y).\end{equation}
The $LLC$ constant of $J$ and $\lambda'$ depend only on each other.  The condition \eqref{three point} is  often called the three-point condition.  Theorem \ref{metric quasicircles} implies that it may be used to characterize quasicircles as well. 
\end{remark}

A rectifiable path $\gamma\colon [a,b] \to X$ is said to be an $l$-chord-arc path, $l \geq 1$, if for every $s \leq t \in [a,b]$, 
$$\lng(\gamma|_{[s,t]}) \leq l d(\gamma(s), \gamma(t)).$$
Similarly, a continuous map $\gamma\colon \sphere^1 \to X$ is called an $l$-chord-arc loop if the following condition holds for all $\theta, \phi \in \sphere^1$.  Let $J_1$ and $J_2$ be the unique subarcs of $\sphere^1$ such that $J_1 \cup J_2 = \sphere^1$ and $J_1 \cap J_2 = \{\theta, \phi\}$, then
$$\min\left\{\lng{J_1}, \lng{J_2}\right\} \leq l d(\gamma(\theta), \gamma(\phi)).$$
A chord-arc path or loop which is parameterized by arc length is an embedding, and the image is an arc or Jordan curve, respectively.  Note that by Theorem \ref{metric quasicircles} and the three-point characterization of the $LLC$ condition given by \eqref{three point}, the image of a chord-arc loop is a quasicircle.  

Ahlfors $2$-regular and $LLC$ metric spaces homeomorphic to a simply connected surface have been classified up to quasisymmetry \cite{B-K}, \cite{me}.  We will need the following statement, which is proven in manner similar to \cite[Theorem 1.2 (iii)]{me}.

\begin{theorem}\label{main} Let $X$ be a metric space homeomorphic to the plane such that $\ovl{X}$ is bounded, Ahlfors $2$-regular, and $LLC$, and such that $\partial{X}$ is a Jordan curve satisfying \eqref{three point}.  Then $X$ is quasisymmetrically equivalent to $\disk$.  The distortion function of the quasisymmetry can be chosen to depend only on the constants associated with the assumptions and the ratio $\diam{X}/\diam{\partial{X}}$.  
\end{theorem}

\begin{proof}   Throughout this proof, ``the data" refers to the Ahlfors $2$-regularity and $LLC$ constants of $\ovl{X}$, the constant associated to $\partial{X}$ by \eqref{three point}, and the ratio $\diam{X}/\diam{\partial{X}}.$  Let $X'$ be the space obtained by gluing two copies of $\ovl{X}$ together along $\partial{X}$.  Then $\ovl{X}$ embeds isometrically in to $X'$, which is homeomorphic to $\sphere^2$.  Furthermore, it is shown in \cite[Section 5]{me}, that $X'$ Ahlfors $2$-regular and $LLC$, with constants depending only on the data.  Bonk and Kleiner's uniformization result for $\sphere^2$ \cite[Theorem 1.1]{B-K} implies that there is a quasisymmetric homeomorphism $f\colon X' \to \sphere^2$ whose distortion function depends only on the data.  By Theorem \ref{LLC is qs invariant} and Remark \ref{3 point}, $f(\partial{X})$ is an $LLC$ Jordan curve in $\sphere^2$ with constants depending only on the data.  The classical theory of conformal welding (cf.\ \cite{Lehto}, \cite{LehtoTeich}) now implies that there is a global quasisymmetric map $g\colon \sphere^2 \to \sphere^2$, with distortion depending only the data, such that $g\circ f(X) = \disk$.  
\end{proof} 

\section{Local uniformization}
\indent

Theorem \ref{main II} states that if $(X,d)$ is a locally Ahlfors $2$-regular (Definition \ref{loc reg definition}) and $LLLC$ (Definition \ref{LLcont}) metric space homeomorphic to a surface, then each point $z \in X$ has a neighborhood which is quasisymmetrically equivalent to the disk.  The following quantitative result immediately implies Theorem \ref{main II}.

\begin{theorem}\label{existence}
Let $(X,d)$ be a proper, $LLLC$, and locally Ahlfors $2$-regular metric space homeomorphic to a surface. Let $K$ be a compact subset of $X$, and $R_K$, $C_K$, and $\Lambda_K$ be the radius and constants associated to $K$ by the assumptions.  Let $z$ be an interior point of $K$ and set 
$$R_0 =\min \{ \max\{R \geq 0 : \ovl{B}(z, R) \subeq K\}, R_K\} > 0.$$ 
Then there exist constants $A_1,A_2 \geq 1$ depending only on $C_K$ and $\Lambda_K$ such that for all $0< R \leq R_0/A_1$, there is a neighborhood $\Omega$ of $z$ such that
\begin{itemize}
\item[(i)] $B\left(z, R/A_2\right) \subeq \Omega \subeq B(z, A_2R)$,
\item[(ii)] there exists an $\eta$-quasisymmetric map $f\colon \Omega \to \disk$, where $\eta$ depends only on $C_K$ and $\Lambda_K$.  
\end{itemize}
\end{theorem}

\subsection{Bounded turning and quasiarcs}\label{bt}
\indent 

As discussed in the introduction, our proof of Theorem \ref{existence} requires that we first give a quasiconvexity result.  To do so, we need a technical result similar to, but weaker than, the following theorem of Tukia \cite{BT}. Let $X \subeq \reals^n$ be endowed with the metric inherited from $\reals^n$.  If $X$ is of bounded turning in itself, then any two points of $X$ can be connected by a quasiarc, i.e., the quasisymmetric image of an interval.  

\begin{definition}\label{cqa def} Let $\ep > 0$ and $M \geq 1$.  An arc $\alpha$ in a metric space $(X,d)$ is an $(\ep,M)$-quasiarc if each pair of points $u,v \in \alpha$ with $d(u,v) \leq \ep$, we have
$$\diam{\alpha[u,v]} \leq M\ep.$$
\end{definition}

\begin{proposition}\label{cqa existence} Let $(X,d)$ be a locally path connected metric space, and $z \in X$. Suppose that there is a radius $R > 0$ and constants $Q, D,$ and $\lambda$ such that $B(z, R)$ is of $\lambda$-bounded turning in $X$, and has relative Assouad dimension at most $Q$ with constant $D$. Then there are  constants $M, N\geq 1$ and $c > 0$, all depending only on $Q, D,$ and $\lambda$, with the following property.  For all pairs of points $x,y \subeq B(z , cR)$ and all $0 < \ep < d(x,y)$, there is an $(\ep,M)$-quasiarc connecting $x$ to $y$ which is contained in $B(x, Nd(x,y)).$ 
\end{proposition}  

For the proof, we need a lemma regarding the extraction of an arc from the image of a path (see Remark \ref{arcs}).  In general, there is not a unique way to do so.  However, in the case that the path is a concatenation of embeddings, there is a canonical choice.

\begin{proposition}\label{arc concatenation} Let $(X,d)$ be a metric space and $n \in \nats$.  For $i =0,\hdots, n$, let $\gamma_i\colon [a_i, b_i] \to X$ be an embedding.  For $i=0, \hdots, n-1$, assume that $\gamma_i(b_i)=\gamma_{i+1}(a_{i+1})$.  Then there is an arc $\alpha$ connecting $\gamma_0(a_0)$ to $\gamma_n(b_n)$ such that $\alpha \subeq\im(\gamma_0 \cdot \hdots \cdot \gamma_n)$ with the following property: if $u, v \in \alpha$, then there are indices $i,j \in \{0,\hdots, n\}$ such that $u \in \im(\gamma_i)$, $v \in \im(\gamma_j)$, and $\alpha[u,v]$ is contained in $\bigcup \im(\gamma_l)$, where $l$ ranges over indices  between $i$ and $j$, inclusively.
\end{proposition}

\begin{proof} We construct the arc via an inductive process.  Let $i_0 = 0$ and $a'_0 = a_0$.  Let $i_1$ be the largest index $i \in \{1, \hdots, n\}$ such that $\im(\gamma_0) \cap \im (\gamma_i) \neq \emptyset$; this is well defined since $\gamma_0(b_0) = \gamma_1(a_1).$  Set  
$$b'_0 = \gamma_0\inv\left(\max \{t \in [a_{i_1}, b_{i_1}] : \gamma_{i_1}(t) \in \im\gamma_0\}\right).$$
Now assume that $k \geq 1$, and that the indices $i_{k-1} < i_{k} \in \{1, \hdots, n\}$ and the interval $[a'_{i_{k-1}}, b'_{i_{k-1}}] \subeq [a_{i_{k-1}}, b_{i_{k-1}}]$ are defined and satisfy
$$\gamma_{i_{k-1}}(b'_{i_{k-1}}) \in \gamma_{i_k}.$$
 Set 
$$a'_{i_k} = \gamma_{i_{k}}\inv (\gamma_{i_{k-1}}(b'_{k-1})) \in [a_{i_k}, b_{i_k}].$$
If $i_k =n$, set $b'_{i_k} = b_n$ and stop.  If not, let $i_{k+1}$ be the largest index $i \in \{i_{k}+1, \hdots, n\}$ such that 
$$\gamma_{i_{k}}([a'_{i_k}, b_{i_k}]) \cap \im(\gamma_i) \neq \emptyset,$$
and define
$$b'_{i_k} = \gamma_{i_k}\inv\left(\max\{t \in [a_{i_{k+1}}, b_{i_{k+1}}] : \gamma_{i_{k+1}}(t) \in \im\gamma_{i_k} \} \right).$$

Since $i_{k} < i_{k+1}$, this process stops after finitely many steps.  Let $i_m = n$ be the final index. Define $\alpha$ to be the image of the concatenation of $\gamma_{i_k}|_{[a'_{i_k}, b'_{i_k}]}$ for $k=0, \hdots, m.$  Then $\alpha$ is an arc with the desired properties.
 \end{proof}

We will also need a notion of a discrete path. For $\ep > 0$, An $\ep$-chain connecting points $x$ and $y$ of $X$ is defined to be a sequence of points $x=x_0, x_1,\hdots, x_n = y$ in $X$ such that $d(x_i, x_{i+1}) \leq \ep$ for each $i=0,\hdots, n-1.$  We will often denote $\ep$-chains using bold face, e.g., $\mbf{x}=x_0, \hdots, x_n.$  In a connected space, any two points may be connected by an efficient chain.  

\begin{lemma}\label{chains} Let $(X,d)$ be a connected metric space and $\ep >0$.  For any pair of points $x, y \in X$, there is an $\ep$-chain $x_0, \hdots, x_n$ connecting $x$ to $y$ that contains an $\ep$-separated set of cardinality at least $n/2$.  
\end{lemma}

\begin{proof} For any $z \in X$, let 
$$S(z):= \bigcup\{w \in X : \ \hbox{there exists an $\ep$-chain from $z$ to $w$}\}.$$
Then $S(z)$ is an open set, and if $S(z) \cap S(w) \neq \emptyset$, then $S(z)=S(w)$.  By connectedness, we see that $S(x)= X$.  If $x_0, \hdots, x_n$ is the $\ep$-chain from $x$ to $y$ of minimal cardinality, then $d(x_i, x_j) \geq \ep$ for all $i=0, \hdots, n-2$ and $j \geq i+2$.  This implies that the set of even-indexed points in the chain is $\ep$-separated. 
\end{proof}

\begin{proof}[Proof of Proposition \ref{cqa existence}]  The basic idea is the following.  Take an $\ep$-chain of minimal cardinality connecting $x$ to $y$.  We may use the bounded turning condition to connect consecutive points of the chain. The resulting concatenation contains an $(\ep, M)$-quasiarc connecting $x$ to $y$, for otherwise we may find a shorter $\ep$-chain.  Unfortunately, it is not true in general that this arc will be contained in a ball around $x$ with controlled radius.  To overcome this, we introduce a ``score" function on $\ep$-chains which balances distance from $x$ with cardinality. 

We now begin the formal proof. As per Remark \ref{arcs}, we may assume with out loss of generality that the bounded turning condition provides arcs rather than arbitrary continua.  

Set 
$$c = \frac{1}{1 + 8\lambda + 2(4D)^{\frac{1}{2Q}}(4\lambda)^{\frac{1}{2}}}.$$

Let $x, y \in B(z, cR)$, and set $d(x,y) = r.$ The bounded turning condition provides an arc $\gamma$ connecting $x$ to $y$ with $\diam{\gamma} \leq \lambda r$.  Let $\ep < r$.    For any $\ep$-chain $\mbf{w}$ in $X$, define the $\ep$-score function
$$\sigma_{\ep}(\mbf{w}) = \sum_{w \in \mbf{w}}1 + \left(\frac{\dist(w, \gamma)}{\ep}\right)^{2Q}.$$
  As $\gamma \subeq B(x , 2\lambda r)$ and $2 \lambda r < 4\lambda cR  < R$, the arc $\gamma$ may be covered by at most $D(4\lambda r/\ep)^Q$ balls of radius $\ep/2$.  By Lemma \ref{chains}, there is an $\ep$-chain $\mbf{w} \subeq \gamma$ connecting $x$ to $y$ containing an $\ep$-separated set of cardinality at least $\card{\mbf{w}}/2.$  Since $\ep$-separated points cannot be contained in a single $(\ep/2)$ ball, we have that  
$$\card{\mbf{w}} \leq 2D(4\lambda r/\ep)^Q.$$

Let $\mathcal{A}_{\ep}(x,y)$ be the set of all $\ep$-chains in $X$ connecting points $x,y \in X$.  Let $ \{z_0, \hdots, z_n\}=\mbf{z} \in \adm_{\ep}(x,y)$ be such that 
$$\sigma_{\ep}(\mbf{z}) \leq \inf_{\mbf{x} \in \adm_{\ep}(x,y)} \sigma_{\ep}(\mbf{x}) + 1.$$
Note that 
$$\inf_{\mbf{x} \in \adm_{\ep}(x,y)} \sigma_{\ep}(\mbf{x}) \geq 2,$$ and so we also have 
$$\sigma_{\ep}(\mbf{z}) \leq 2 \inf_{\mbf{x} \in \adm_{\ep}(x,y)} \sigma_{\ep}(\mbf{x}).$$

Set
$$H := \max_{i=0,\hdots, n} \dist(z_i, \gamma).$$
Then 
$$ \left(\frac{H}{\ep}\right)^{2Q} \leq \sigma_{\ep}(\mbf{z}) \leq 2 \sigma_{\ep}(\mbf{w}) \leq 4D\left(\frac{4\lambda r}{\ep}\right)^Q.$$ 
As a result, we have 
$$H \leq (4D)^{1/(2Q)} (4\lambda r \ep)^{1/2}.$$
Since $r\ep < (2cR)^2$,  we have for $i = 0, \hdots, n,$
$$d(z, z_i) \leq d(z, x) + \diam{\gamma} + H < cR + 2\lambda cR +  (4D)^{1/(2Q)} (4\lambda)^{1/2}2cR < R.$$
For each $i = 0,\hdots, n,$ the bounded turning condition provides an embedding $\gamma_i \colon [0,1] \to X$ connecting $z_i$ to $z_{i+1}$ with $\diam(\im\gamma_i) \leq \lambda \ep.$ 

Note that if $p \in \im\gamma_i$ for some $i =0, \hdots, n$, then
$$d(x, p) \leq \dist(p, \mbf{z}) + H + \diam(\gamma) \leq \lambda \ep +  (4D)^{\frac{1}{2Q}} (4\lambda r \ep)^{\frac{1}{2}} + \lambda r.$$
Since $\ep < r$, this implies $\im{\gamma_{i}} \subeq {B}(x, Nr)$ for $N = 2\lambda + (4D)^{\frac{1}{2Q}}(4\lambda)^{\frac{1}{2}}$.

We now make a claim which will quickly imply the desired result.  Let $i, j \in \{0, \hdots, n\}$, and suppose that  $u \in \im{\gamma_i}$ and $v \in \im{\gamma_j}$ are points such that $d(u,v) \leq \ep$. Then there is an integer $M_0$, depending only on $Q,D,$ and $\lambda$,  such that $|j-i| \leq M_0$.

Suppose that the claim is true.  We may extract an arc $\alpha$ connecting $x$ to $y$ from the image of the concatenation $\gamma_0 \cdot \hdots \cdot \gamma_n$, as in Proposition \ref{arc concatenation}.  As $\im\gamma_i \subeq B(x, Nr)$ for $i=0, \hdots, n$, we have that $\alpha \subeq B(x, Nr)$.  We now show that $\alpha$ is an $(\ep, M_0\lambda)$-quasiarc.  Let $u, v \in \alpha$ be such that $d(u,v) \leq \ep$.  We may find indices $i, j \in \{0,\hdots, n\}$ such that $u \in \im\gamma_i$, $v \in \im\gamma_j$, and $\alpha[u,v]$ is contained in $\bigcup \im(\gamma_l)$, where $l$ ranges over indices  between $i$ and $j$, inclusively.  Without loss of generality, assume $i \leq j$. The claim yields
$$\diam(\alpha[u,v]) \leq  \diam\left( \bigcup_{l=i}^j \im\gamma_l\right) \leq \sum_{l=i}^j \diam(\im(\gamma_l)) \leq M_0\lambda \ep,$$
as desired.

We now prove the claim.  Let $i, j \in \{0, \hdots, n\}$, and suppose that  $u \in \im{\gamma_i}$ and $v \in \im{\gamma_j}$ are points such that $d(u,v) \leq \ep$.  Without loss of generality, we assume that $i \leq j$.  Note that as $x \in B(z, cR)$,
$$\im(\gamma_i) \subeq B(z_i, 2\lambda \ep) \subeq B(x, (N + 2\lambda)r) \subeq B(z, R),$$
and so there is a cover of $\im\gamma_i$ by no more than $D(2\lambda)^Q$ balls of radius $\ep$.  The same holds for $\im{\gamma_j}$.  Applying Lemma \ref{chains} to $\im{\gamma_i}$ and $\im\gamma_j$ provides $\ep$-chains $\mbf{w}_u$ and $\mbf{w}_v$ connecting $z_i$ to $u$ and $v$ to $z_{j+1}$, each of cardinality no greater than $2D(2\lambda)^Q.$  It follows that
$$\mbf{w}:= \{z_0, \hdots, z_{i-1} \} \cup \mbf{w}_u \cup \mbf{w}_v \cup \{z_{j+2}, \hdots, z_{n}\}$$
is an $\ep$-chain.   

\begin{figure}[h]\label{btgood}
\begin{center}
\psfrag{U}{$u$}
\psfrag{V}{$v$}
\psfrag{zi}{$z_i$}
\psfrag{zi1}{$z_{i+1}$}
\psfrag{zj}{$z_j$}
\psfrag{zj1}{$z_{j+1}$}
\psfrag{wu}{$\mbf{w}_u$}
\psfrag{wv}{$\mbf{w}_v$}
\psfrag{gi}{$\im\gamma_i$}
\psfrag{gj}{$\im\gamma_j$}
\includegraphics[height=.35\textwidth, width=.75\textwidth]{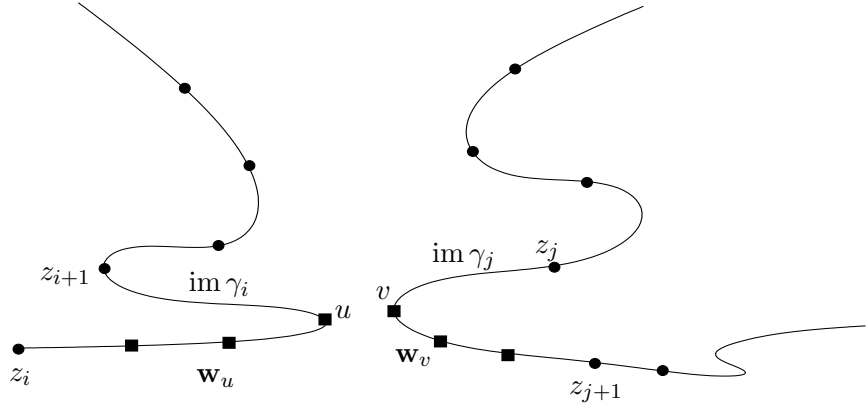}
\caption{The shortcut chains $\mbf{w}_u$ and $\mbf{w}_v$}
\end{center}
\end{figure}

We now use the inequality
$$\sigma_{\ep}(\mbf{z}) \leq \sigma_{\ep}(\mbf{w}) + 1.$$
Canceling the points where $\mbf{z}$ and $\mbf{w}$ agree, this inequality simplifies to
\begin{equation}\label{first chain ineq} 
\sum_{l=i}^{j+1} 1 + \left(\frac{\dist(z_l, \gamma)}{\ep}\right)^{2Q} \leq  \sigma_{\ep}(\mbf{w}_{u} \cup \mbf{w}_v) + 1.
\end{equation}
Let $\Theta = \dist(z_i, \gamma)/\ep,$ $m=j-i+1$, and $A=\card(\mbf{w}_u \cup \mbf{w}_v)-1$.  Then $A$ depends only on $Q,D,$ and $\lambda$.  Note that if $a$ and $b$ are consecutive points of an $\ep$-chain, then
$$\left| \frac{d(a, \gamma)}{\ep} - \frac{d(b, \gamma)}{\ep}\right| \leq 1.$$
Thus \eqref{first chain ineq} implies 
\begin{equation}\label{chain ineq} 
\sum_{l=0}^{m}1 + \left(\max\left \{ \Theta - l, 0 \right \}\right)^{2Q} \leq 1 + \sum_{l=0}^{A} 1 + (\Theta + l)^{2Q}.
 \end{equation}
 
 In order to show that $m$, and hence $j-i$, is bounded above by a constant depending only on $Q, \lambda,$ and $D$, we must analyze a few cases.  Let 
$$\Theta_0 := \sup \left \{ \Theta: 1 + (A+1)(1 + (\Theta + A)^{2Q}) - \frac{\Theta^{2Q+1} - (\Theta/2)^{2Q+1}}{2Q+1}  \geq 0 \right \}.$$
Then $\Theta_0$ is finite and depends only $Q,D$ and $\lambda$.  

\bigskip
\noindent\textit{Case 1: $\Theta \leq \Theta_0$.}  By \eqref{chain ineq}, we have 
\begin{equation}\label{case 1}
m \leq  \sum_{l = 0}^m 1 + \left(\max\left \{ \Theta - l, 0 \right \}\right)^{2Q} \leq 1+  \sum_{l=0}^{A} 1 + (\Theta+l)^{2Q} \leq 1+ (A+1)(1+ (\Theta_0+A)^{2Q}).
\end{equation}   
This provides the desired bound.

\bigskip
\noindent\textit{Case 2: $\Theta > \Theta_0$ and $m \geq \Theta/2.$}  We have
$$\sum_{l=0}^{m}1 + \left(\max\left \{ \Theta - l, 0 \right \}\right)^{2Q} \geq m + \sum_{l=0}^{\lfloor \Theta/2 \rfloor} (\Theta -l)^{2Q} \geq m+\int_{0}^{\lfloor \Theta/2 \rfloor +1} (\Theta-l)^{2Q} dl.$$
This, combined with \eqref{chain ineq}, shows that 
$$m + \frac{\Theta^{2Q+1} - (\Theta/2)^{2Q+1}}{2Q+1} \leq 1 + (A+1)(1 + (\Theta + A)^{2Q}).$$
By the definition of $\Theta_0$, this yields that $m < 0$, a contradiction.  

\bigskip
\noindent\textit{Case 3: $\Theta > \Theta_0$ and $m \leq A$.} As $A$ depends only on $Q, D$, and $\lambda$, we already have the desired bound. 

\bigskip
\noindent\textit{Case 4: $\Theta > \Theta_0$ and $A < m \leq \Theta/2.$}  Note that by definition, $A > 1$.  We have 
$$m\left(\frac{\Theta}{2} \right)^{2Q} \leq \sum_{l=0}^m 1 + \max \left\{ (\Theta -l)^{2Q}, 0 \right\},$$
as well as
$$1+ \sum_{l=0}^{A} 1 + (\Theta+l)^{2Q} \leq 1+ (A+1)(1+(\Theta+A)^{2Q}) \leq 8A(2\Theta)^{2Q}.$$
From these inequalities and \eqref{chain ineq}, we see that $m \leq 8A(4)^{2Q},$ as desired. 

 Combining these cases shows that $m$ is bounded above by a constant depending only on $Q, \lambda$, and $D$, which completes the proof of the claim.
 \end{proof}

\subsection{Quasiconvexity}

\begin{theorem}\label{quasiconvexity main work} Let $(X,d)$ be a proper metric space, and $z \in X$.  Suppose that $U \subeq X$ is a neighborhood of $z$ homeomorphic to the plane $\reals^2$ and that there are constants $C, \Lambda, M, N \geq 1$ and a radius $R>0$ such that  
\begin{itemize}
\item[(i)] $B(z, 4NR) \subeq U$ and $4NR \leq \diam(U),$
\item[(ii)] $U$ is relatively Ahlfors $2$-regular with constant $C$ (see Definition \ref{rel reg definition}),
\item[(iii)] $U$ is relatively $\Lambda$-linearly locally contractible (see Definition \ref{rel LLcont}),
\item[(iv)] for all $x,y \in U$ and $0 < \ep < d(x,y)$, there is an $(\ep, M)$-quasiarc connecting $x$ to $y$ inside of $B(x, Nd(x,y))$ (see Definition \ref{cqa def}).   
\end{itemize}
Then there exists a constant $L \geq 1$ depending only on $C, \Lambda, M,$ and $N$ such that each pair of points $x, y \in B(z, R)$ may be joined by an arc $\gamma$ in $X$ such that $\lng(\gamma) \leq L d(x,y)$.  
\end{theorem}

\begin{remark}\label{redundant} For any pair of points $x, y \in U$,  setting $\ep = d(x,y)/2$ in assumption $(iv)$ above provides \textit{some} arc connecting $x$ to $y$ inside the ball $B(x, Nd(x,y)).$ Of course, a similar connection (with a different constant) is provided by assumption $(iii)$, as in Proposition \ref{LLLC implication}.  
\end{remark}

Proposition \ref{cqa existence}, Theorem \ref{quasiconvexity main work}, and a covering argument show the following corollary.

\begin{corollary} A proper, locally Ahlfors $2$-regular, and $LLLC$ metric space homeomorphic to a surface is locally quasiconvex.
\end{corollary} 

We recall that for any continuous map $\gamma \colon \sphere^1 \to \reals^2$ and point $z\in \reals^2\bslash \im{\gamma}$, the index $\ind(\gamma, z)$ of $\gamma$ with respect to $z$ is an integer which indicates the number of times $\gamma$ ``wraps around $z$", taking orientation into account.  For a full definition and description, see for example \cite[Chapter 4.2]{Burckel}.   The fundamental property of the index is that it  behaves well under homotopies.  If $H \colon \sphere^1 \times [0,1] \to \reals^2$ is continuous, and $z \notin \im{H}$, then $\ind(H(\cdot, 0), z) = \ind(H(\cdot, 1), z).$ Moreover, the index is additive under concatenation: if $\gamma_1, \gamma_2 \colon \sphere^1 \to \reals^2$ are loops and $z \in \reals^2$ is a point not in the image of either loop, then $\ind(\gamma_1 \cdot \gamma_2, z) = \ind(\gamma_1, z) + \ind(\gamma_2, z)$.  

Given a Jordan curve $J$ in the plane, we may find a parameterization $\gamma$ of $J$ such that 
$\ind(\gamma, z) = 1$ for all $z \in \ins(J)$ and $\ind(\gamma,z) = -1$ for all $z \in \reals^2\bslash (J \cup \ins(J)).$ On the other hand, if the image of a continuous map $\gamma \colon \sphere^1 \to \reals^2$ is an arc, then $\ind(\gamma, z) = 0 $ for all $z \notin \im(\gamma)$.  

Assumptions $(ii)$ and $(iii)$ of Theorem \ref{quasiconvexity main work} provide convenient criteria for showing that there is a controlled homotopy between certain loops.  This trick will play a role in the proof of Theorem \ref{quasiconvexity main work} as well as the rest of the proof of Theorem \ref{existence}. 

\begin{lemma}\label{homotopy trick}  Let $(X,d)$ be a metric space, and let $\Lambda \geq 1$ and $\del>0$.  Suppose that $U \subeq X$ is a subset such that for all $x \in U$ and $0 < r \leq 2\del(\Lambda+1)$ the ball $B(x,r)$ is contractible inside the ball $B(x, \Lambda r)$.  Let $\alpha$ and $\beta$ be continuous maps of $\sphere^1$ into $U$.  If there exists a cyclically ordered set $\{\theta_1, \hdots , \theta_n\} \subeq \sphere^1$ such that $d(\alpha(\theta_i),\beta(\theta_i)) \leq \del$ for $i=1,\hdots, n,$ and 
$$\max\left\{\diam(\alpha([\theta_i,\theta_{i+1}])), \diam(\beta([\theta_i, \theta_{i+1}]))\right\} \leq \del$$
for $i = 1, \hdots, n,$ $\modu n$, then there is a homotopy $H \colon \sphere^1 \times [0,1] \to \nbhd_{2\Lambda\del(\Lambda+1)}(\im{\beta})$ such that $H(\cdot, 0) = \alpha$ and $H(\cdot, 1) = \beta$.  
\end{lemma}

\begin{proof} As in the proof of Proposition \ref{LLLC implication}, we may assume that for every pair of points $x,y \in U$ with $d(x,y) < \del(\Lambda+1)$, there is a path connecting $x$ to $y$ inside $B(x, 2\Lambda d(x,y))$.  

Throughout this proof, consider the indices $\{1, \hdots, n\}$ modulo $n$.  Fix $i \in \{1, \hdots, n\}$, and define $a_i = \alpha(\theta_i)$ and $b_i=\beta(\theta_i)$. We may find a path $\gamma_i \colon [0,1] \to B(b_i, 2\Lambda \del)$ such that $\gamma_i(0)=a_i$ and $\gamma_i(1) = b_i.$  

Define a subset $L_i$ of $\sphere^1 \times [0,1]$ by $L_i = [\theta_i, \theta_{i+1}] \times [0,1]$.  The strip $L_i$ is bounded by the curve
$$l_i := ([\theta_i, \theta_{i+1}] \times \{0\}) \cup (\{\theta_{i+1}\} \times [0,1]) \cup ([\theta_i, \theta_{i+1}] \times \{1\}) \cup (\{\theta_i\} \times [0,1]).$$
Define a continuous map $g_i: l_i \to X$ by 
$$g_i(\theta, s) = \begin{cases} 
			\alpha(\theta) & (\theta, s) \in [\theta_i, \theta_{i+1}] \times \{0\}, \\
			\gamma_{i+1}(s)  & (\theta,s) \in \{\theta_{i+1} \}\times [0,1], \\ 
			\beta(\theta) & (\theta,s) \in [\theta_i, \theta_{i+1}] \times \{1\}, \\
			\gamma_{i}(s) & (\theta,s) \in \{\theta_{i} \}\times [0,1]. \\
		    \end{cases}$$
As 
$$\gamma_{i}([0,1]) \subeq B(b_i, 2\Lambda \del), \mand  \gamma_{i+1}([0,1]) \subeq B(b_{i+1}, 2\Lambda \del), $$
and 
$$ \alpha([\theta_i, \theta_{i+1}])  \subeq B(a_i, 2\del), \mand \beta([\theta_i, \theta_{i+1}]) \subeq B(b_i, 2\del),$$
we see that $\im(g_i) \subeq B(b_i, 2\del(\Lambda+1)).$  The contractibility assumption now provides a homotopy 
$$H_i \colon B(b_i, 2\del(\Lambda+1)) \times [0,1] \to B(b_i, 2\Lambda\del(\Lambda+1))$$
such that $H_i(\cdot, 0)$ is the identity map and $H_i(\cdot, 1)$ is a constant map.

\begin{figure}[h]\label{homotopy}
\begin{center}
\psfrag{0}{$0$}
\psfrag{1}{$1$}
\psfrag{ti}{$\theta_i$}
\psfrag{ti1}{$\theta_{i+1}$}
\psfrag{Li1}{$L'_i$}
\psfrag{Li}{$L_i$}
\psfrag{gi}{$\gamma_i([0,1])$}
\psfrag{gi1}{$\gamma_{i+1}([0,1])$}
\psfrag{a}{$\alpha([\theta_i, \theta_{i+1}])$}
\psfrag{b}{$\beta([\theta_i, \theta_{i+1}])$}
\psfrag{li}{$l_i$}
\psfrag{lix}{$l_i \times \{0\}$}
\includegraphics[width=.75\textwidth]{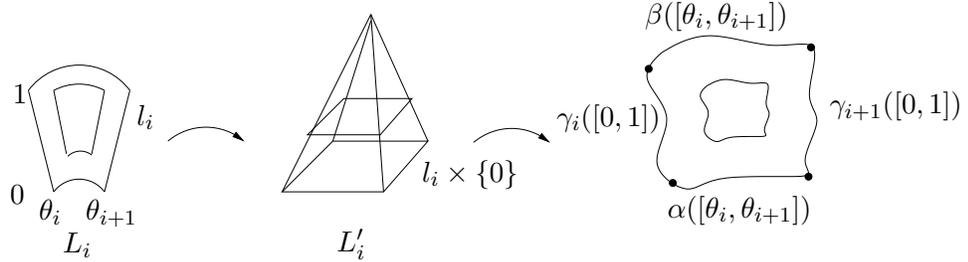}
\caption{The homotopy trick}
\end{center}
\end{figure}

Let $L'_i$ be the quotient of $l_i \times [0,1]$ obtained by identifying all points in $l_i \times \{1\}$.  Then there is a homeomorphism $f_i \colon L_i \to L'_i$ which maps $l_i$ to the copy of $l_i \times \{0\}$ in $L'_i$. Denote 
$$f_i(\theta, s) = [l(\theta, s), t(\theta, s)] \in L'_i.$$
Define $G_i \colon L_i \to X$ by 
$$G_i(\theta, s) = H_i(g_i(l(\theta, s)), t(\theta, s)) \subeq B(b_i, 2\Lambda\del(\Lambda+1)).$$
This is well defined and continuous because $H_i(\cdot, 1)$ is a constant map.  Furthermore, $G_i|_{l_i} = g_i$, since $H(\cdot, 0)$ is the identity mapping and $f_i$ maps $l_i$ to $l_i \times \{0\}$.  As a result, $G_i$ agrees with $G_{i+1}$ on $L_i \cap L_{i+1} = \{\theta_{i+1}\}\times [0,1].$  Thus we may form a continuous map $G \colon \sphere^1 \times [0,1] \to X$ by setting $G=G_i$ on $L_i$. From the definition of $G_i$ and $g_i$, we see that $G(\cdot, 0) = \alpha$ and $G(\cdot, 1) = \beta,$ as desired.
\end{proof}

Our main tool in the proof of Theorem \ref{quasiconvexity main work} is the following consequence of a co-area formula for Lipschitz maps of metric spaces.  See \cite[Prop. 3.1.5]{Ambrosio} for a proof and discussion.

\begin{theorem}\label{co-area} Let $(X,d)$ be a metric space, and let $E \subeq X$ be a continuum. For $t>0$, set 
$$E_t := \{x \in X: \dist(x,E) < t\} \quad \hbox{and} \quad L_t:=\{x \in X: \dist(x,E)=t\}.$$
There exists a universal constant $\omega$ such that if $T \geq 0$ and 
$$\int_0^T \Hdim^1(L_t) \ dt \leq \omega \Hdim^2(E_T \cup L_T).$$
\end{theorem}

A continuum $E$ in a metric space need not be locally connected, and hence is not necessarily arc-connected.  The additional condition that $\Hdim^1(E) < \infty$ provides this property.  Moreover, the arcs can be chosen so that their length is majorized by the $\Hdim^1(E)$.  See \cite[Section 15]{QuantTop} for a proof of the following well-known result.

\begin{proposition}\label{continuum length} Let $E$ be a continuum in a metric space $(X,d)$ such that $\Hdim^1(E) <\infty$, and let $x, y \in E$.  Then there exists an embedding $\gamma \colon [0,1] \to E$ connecting $x$ to $y$ such that $\lng{\gamma}\leq \Hdim^1(E).$
\end{proposition}

We will also need some facts from elementary topology.  The following statement follows from the compactness of $\sphere^1$ and the fact that a continuous map on a compact set is uniformly continuous.

\begin{proposition} \label{jordan curve segments}   Let $(X,d)$ be a metric space, and $\alpha:\sphere^1 \to X$ a continuous map.  For each $t > 0$, there exists a finite subset $\{\theta_1, \hdots, \theta_n\} \subeq \sphere^1$ in cyclic order such that 
$\diam(\alpha([\theta_i, \theta_{i+1}])) \leq t $ for $i = 1,\hdots, n \modu{n}$. 
\end{proposition}

By Schoenflies' Theorem, a Jordan curve in the plane is the topological boundary of its inside.  The inside of a Jordan curve can be used to approximate a compact, connected subset of the plane \cite[Page 347]{Burckel}

\begin{proposition}\label{Jordan curve approximation} Let $K$ be a compact and connected subset of $\reals^2$.  For any open set $V \supeq K$, there is a Jordan curve $\alpha$ such that $K \subeq \ins{\alpha}$ and $\alpha \subeq K.$
\end{proposition}

The final topological fact we need is also well-known.  Here we consider the sphere $\sphere^2$ as the one point compactification of $\reals^2$, with the added point labeled $\infty$.

\begin{proposition}\label{connected complement} Let $V$ be an open and connected subset of $\reals^2$ such that $\reals^2 \bslash V$ is a continuum.  Then $V \cup \{\infty\} \subeq \sphere^2$ is homeomorphic to the plane.  Moreover, the topological boundary of $V$ in $\reals^2$ is connected.  
\end{proposition}

\begin{proof}  The first statement is an immediate corollary of Proposition \ref{complement prop}, which is proven independently.  The second statement is elementary; it follows, for example, from \cite[Exercise 1.5]{Burckel} and the Riemann mapping theorem. 
\end{proof}

\begin{proof}[Proof of Theorem \ref{quasiconvexity main work}]
Let $x, y \in B(z, R)$, and set $r = d(x,y)$, and 
\begin{equation} \label{qconvex ep def} \ep = \frac{r}{32MN\Lambda(\Lambda+1)}.\end{equation}
Since $\ep < r$, there is an $(\ep, M)$-quasiarc $\gamma$ connecting $x$ to $y$ inside of $B(x, Nr) \subeq U.$ 

Let $\ul{\gamma}\colon [0,1] \to X$ be an embedding parameterizing $\gamma$. Define
$$a = \max\{s \in [0,1]: \ul{\gamma}(s) \in \ovl{B}(x, r/8N) \} \mand b=\min\{s \in  [a,1]: \ul{\gamma}(s) \in \ovl{B}(y, r/8N)\}.$$
Define $E = \ul{\gamma}([a,b]).$ Since $d(x,y) = r > r/8N$, we see that $a < 1$. As $\gamma$ is connected, this implies that $d(\gamma(a), x)= r/8N$.  Thus we have $\dist(x,E) = r/8N$.  Similarly $\dist(y, E) = r/8N$.  For $t> 0$, set
$$E_t := \{x \in X: \dist(x,E) \leq t\} \quad \hbox{and} \quad L_t:=\{x \in X: \dist(x,E)=t\}.$$ Then  
$$E_{r/8N} \cup L_{r/8N} \subeq B(x, 2Nr).$$

\begin{figure}[h]\label{quasiconvex1}
\begin{center}
\psfrag{r8n}{$r/8N$}
\psfrag{E}{$E$}
\psfrag{x}{$x$}
\psfrag{y}{$y$}
\psfrag{Lt}{$L_t$}
\includegraphics[width=.75\textwidth]{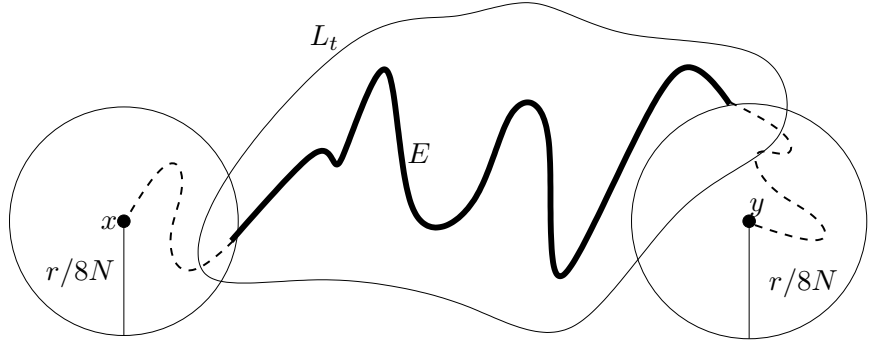}
\caption{Connecting $B(x, r/8N)$ and $B(y, r/8N)$}
\end{center}
\end{figure}

Since $x \in U$, and $2Nr \leq 4NR \leq \diam(U)$, assumptions $(i)$ and $(ii)$ along with Theorem \ref{co-area} yield 
$$\int_0^{r/8N} \Hdim^1(L_t) \ dt \leq \omega\Hdim^2(B(x, 2Nr)) \leq 4C\omega N^2r^2.$$
From this we may estimate that for any $s>8N$,  
\begin{equation}\label{measure estimate}
|\{ t\in [0, r/s] : \Hdim^1(L_t) < 8C\omega N^2sr\}| \geq \frac{r}{2s}. 
\end{equation}

We claim that there is a constant $s_0>8N$, depending only on  $\Lambda, M,$ and $N$, such that for all $t \in [0, r/s_0]$, there is a connected subset of $L_t$ which intersects both $B(x,r/4)$ and $B(y, r/4).$  

Suppose that the claim is true.  The measure estimate \eqref{measure estimate} ensures that there is some $t_0 \in [0, r/s_0]$ such that the level set $L_{t_0}$ satisfies $\Hdim^1(L_{t_0}) \leq 8C\omega N^2 s_0 r.$  Since $X$ is proper, $L_{t_0}$ is compact, and so the closure of the connected subset of $L_{t_0}$ intersecting both $B(x,r/4)$ and $B(y, r/4)$ is a continuum of controlled $\Hdim^1$-measure.  Proposition \ref{continuum length} now provides a path connecting the balls of length $Lr$, where $L$ depends only on $C, \Lambda, M,$ and $N$. The desired path connecting $x$ to $y$ is now easily constructed using an inductive process; see \cite[Lemma 3.4]{ConfDim} for more details.  

We proceed with the proof of the claim.  Let 
$$s_0 = 5r/\ep = 160\Lambda MN(\Lambda+1),$$ and let $t \in [0, r/s_0].$  Consider the set 
$$F_t:= \{w \in U: d(w, E) > t\}.$$
 As $t < r/8N$ and $d(x, E)= r/8N$, the point $x$ must belong to some component $A$ of $F_t$.  

Since $E_t \cup L_t$ is bounded and $X$ is proper, the set $U \bslash F_t$ is compact.  It follows from the fact that $U$ is one-ended that there is a unique component of $F_t$ whose closure is not a compact subset of $U$.  We will first show that $A$ must be this component.  Towards a contradiction, suppose that $\cl{A}$ is a compact subset of $U$.  Then by Proposition \ref{Jordan curve approximation}, there exists a embedding $\alpha:\sphere^1 \to U$ such that $\im(\alpha) \subeq \nbhd_{t/N}(A)$ and $A$ is contained in the inside of the Jordan curve $\im(\alpha).$  

By Proposition \ref{jordan curve segments}, there is a cyclically ordered set $\{\theta_1,\hdots,\theta_n\} \subeq \sphere^1$ such that 
\begin{equation}\label{alpha diam 1}
\diam( \alpha([\theta_i, \theta_{i+1}]) \leq t
\end{equation}
for $i=1, \hdots, n$, $\modu{n}$.  Without loss of generality, we may assume that 
$$0 = \theta_1 < \hdots < \theta_n < 2\pi.$$

\begin{figure}[h]\label{quasiconvex2}
\begin{center}
\psfrag{E}{$E$}
\psfrag{x}{$x$}
\psfrag{Lt}{$L_t$}
\psfrag{ai}{$\alpha(\theta_i)$}
\psfrag{ei}{$e_i$}
\psfrag{lt}{$< 2t$}
\includegraphics[width=.75\textwidth]{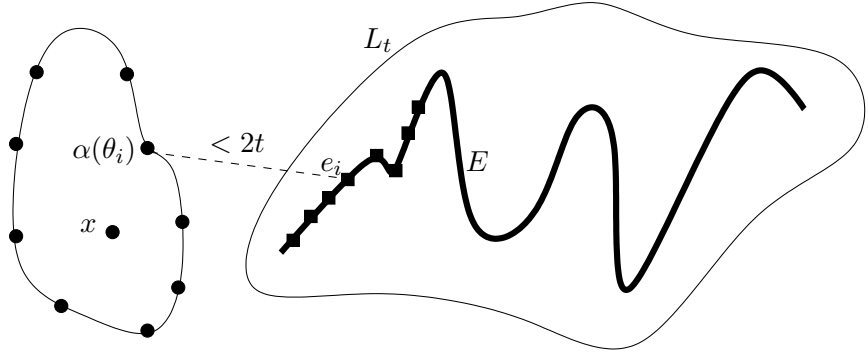}
\caption{The points $\alpha(\theta_i)$ and their partners $e_i$}
\end{center}
\end{figure}

If $w \in \im(\alpha)$, then there is a point $a \in A$ such that $d(w,a) < t/N$.  By Remark \ref{redundant} there is an arc connecting $w$ to $a$ inside $B(w, t)$.   As $a\in A$, $w \notin A$, and $A$ is a component of $F_t$, this path must intersect $L_t$. This implies that  $\dist(w, E) < 2t.$  Thus for each parameter $\theta_i$, we may find a point $e_i \in E$ such that $d(\alpha(\theta_i), e_i) < 2t.$  Note that the set $\{e_1, \hdots, e_n\}$ need not be linearly ordered with respect to the given parameterization $\ul{\gamma}$.  For $i=1,\hdots, n$, set $t_i = \ul{\gamma}\inv(e_i) \in [a,b]$.  For $i=1, \hdots, n-1$, define $\rho_i:[\theta_i, \theta_{i+1}] \to \reals$ to be the unique linear map satisfying $\rho_i(\theta_i) = t_i$ and $\rho_i(\theta_{i+1})=t_{i+1}$, and set $\rho_n: [0, 2\pi-\theta_n] \to \reals$ to be the unique linear map satisfying $\rho_n(0)=t_n$ and $\rho_n(2\pi-\theta_n) = t_1$.  Define $f\colon \sphere^1 \to E$ by 
$$f(\theta) = \begin{cases}
		\ul{\gamma} \circ \rho_i(\theta) &  \theta \in [\theta_i, \theta_{i+1}], i =1, \hdots, n-1,\\
		\ul{\gamma} \circ \rho_n(\theta - \theta_n) & \theta \in [\theta_n, 2\pi). \\
		\end{cases}$$ 
Then for $i= 1,\hdots ,n$ $f|_{[\theta_i, \theta_{i+1}]}$ is an embedding parameterizing $E[e_i, e_{i+1}].$  Thus $\im{f} \subeq E$.  Since $E$ is an arc and $x \notin \im{f}$, this implies that $\ind(f, x) = 0$.  On the other hand, $\ind(\alpha, x) \neq 0$, since $x \in A \subeq \ins(\im\alpha).$ 

To reach a contradiction, we will apply Lemma \ref{homotopy trick} to $\alpha$ and $f$, using $\del = M\ep.$  By the definition of $\ep$, we have 
$$2M\ep(\Lambda+1) = \frac{r}{16\Lambda N}\leq \frac{R}{8\Lambda N} \leq \diam(U).$$
Thus by assumption $(iii)$, each ball $B(x, \rho)$ with $x \in U$ and $\rho \leq 2M\ep(\Lambda+1)$ is contractible inside $B(x, \Lambda \rho).$  This verifies the first hypothesis of Lemma \ref{homotopy trick}.   Consider  the points $\{\alpha(\theta_1), \hdots, \alpha(\theta_n)\} \subeq \im{\alpha}$ and $\{e_1, \hdots, e_n\} \subeq \im{f}$.  We have already seen that $d(\alpha(\theta_i), e_i) < 2t < \ep$ for $i =1,\hdots, n$, and that 
$$\diam( \alpha([\theta_i, \theta_{i+1}]) \leq t \leq \ep$$
for $i=1,\hdots, n$, $\modu n$.  This implies that $d(e_i, e_{i+1}) < 5t \leq \ep$ for $i=1, \hdots, n$, $\modu n$.  Since $\gamma$ is an $(\ep, M)$-quasiarc, 
$$\diam(f([\theta_i, \theta_{i+1}])) \leq M\ep.$$
These statements verify the remaining hypotheses of Lemma \ref{homotopy trick}, and we conclude that $f$ is homotopic to $\alpha$ inside $\nbhd_{2M\Lambda \ep(\Lambda+1)}(\im{f}).$  Since $\im{f} \subeq E$, our choice of $\ep$ and the fact that $\dist(x, E) = r/8N$ show that the tracks of the homotopy do not hit $x$.  This contradicts that $\ind(f, x) \neq \ind(\alpha, x)$.  Thus we have shown that the $x$-component of $F_t$, which we have called $A$, is the unique component of $F_t$ whose closure is not a compact subset of $U$.  
 
 Let $\{W_i\}_{i \in I}$ be the collection of components of $U \bslash A$.  Since $E \subeq U\bslash A$ is connected, there is one such component $W_{i_0}$ which contains $E$.  Note that $A$ is open in $U$ since $U$ is locally connected and $F_t$ is open \cite[Theorem 25.3]{Munkres}.  Thus $U\bslash A$ is closed in $U$, and so $W_{i_0}$ is closed in $U$ as well, as it is a component of a closed set.  In fact, $W_{i_0}$ is a compact subset of $U$, as follows. Since $X$ is proper and $U$ is homeomorphic to the plane, there is a compact set $K \subeq U$ such that $U\bslash K$ is connected and $E_t \cup L_t \subeq K$.  Then $U \bslash K$ must be contained in some component of $F_t$.  Since $U\bslash K$ does not have compact closure in $U$, we must have $U\bslash K \subeq A$.  Thus $W_{i_0} \subeq K$, which implies that $W_{i_0}$ is compact.  

Let 
$$V = A \cup \{W_i\}_{i \neq i_0} = U \bslash W_{i_0}.$$
Since $A$ is connected, each $W_i$ is connected, and $\cl(A) \cap W_i \neq \emptyset$ for each $i$, we see that $V$ is connected.  Let $U \cup \{\infty\}$ denote the one-point compactification of $U$, which is homeomorphic to the sphere $\sphere^2$.  Since $W_{i_0}$ is a continuum, Lemma \ref{connected complement} implies that $V \cup \{\infty\}$, considered as a subset of $U \cup \{\infty\}$, is homeomorphic to the plane, and that the topological boundary of $V$ in $U$ is connected.  

For the remainder of this proof only, for any subset $S \subeq U$, we denote the topological closure of $S$ in $U$ by $\cl(S)$ and the topological boundary of $S$ in $U$ by 
$$\partial{S} = \cl(S) \cap \cl(U \bslash S).$$  

  
We now claim that $\partial{V} \subeq L_t$.  Since $V=U \bslash W_{i_0}$, we see that $\partial{V} = \partial{W_{i_0}}$.  As $W_{i_0}$ is a component of $U\bslash A$, we have 
$$\partial{V} = \partial{W_{i_0}} \subeq \partial{A}.$$ 
Since $A$ is a component of $F_t$, it is relatively closed in $F_t$, and so $\partial{A} \cap F_t = \emptyset.$  However, the continuity of the distance function implies that $\partial{A} \cap E_t = \emptyset,$ and so $\partial{V} \cap U\subeq \partial{A} \cap U \subeq L_t$.  

Since $d(x, E) = r/8N$, Remark \ref{redundant} provides a path from $x$ to $E$ inside $B(x, r/4)\subeq U$. Since $x$ is in $A$ and $E \subeq W_{i_0}$, this path must intersect $\partial{V}$ at some point $x'$.  Similarly there is a point $y' \in B(y, r/4) \cap \partial{V}$.  This shows that $\partial{V}$ is a connected subset of $L_t$ which connects $B(x, r/4)$ to $B(y, r/4)$, proving the claim, and completing the proof.
\end{proof}

The measure estimate \eqref{measure estimate} actually allows us to prove more than just local quasiconvexity.   We refer to \cite{Acta} and \cite{LAMS} for the definition of the modulus of a curve family and the basic facts regarding modulus, Loewner spaces, and Poincar\'e inequalities. 

\begin{corollary}\label{modulus cor} Assume the hypotheses of Theorem \ref{quasiconvexity main work}.  There is a constant $L \geq 0$ depending only on $C, \Lambda, M,$ and $N$ such that the following statement holds.  Let $x,y \in B(z, R)$ and set $r=d(x,y)$. Then the $2$-modulus of the path family connecting $B(x, r/4)$ to $B(y, r/4)$ is at least $L$.  
\end{corollary}

\begin{proof}  Let $s_0$ be as in the proof of Theorem \ref{quasiconvexity main work}, and set 
$$\mathcal{G} = \{t \in [0, r/{s_0}] : \Hdim^1(L_t) \leq 8C\omega N^2s_0r\}.$$
The measure estimate \eqref{measure estimate} shows that 
$$\Hdim^1(\mathcal{G}) \geq \frac{r}{2s_0}.$$
By the proof of Theorem \ref{quasiconvexity main work} and the arc extraction discussed in Remark \ref{arcs}, for each $t \in \mathcal{G}$ we may find an arc $\gamma_t \subeq L_t$ connecting $B(x, r/4)$ to $B(y,r/4)$.  Thus it suffices to show that the $2$-modulus of the family $\{\gamma_t: t\in \mathcal{G}\}$ is bounded away from zero.  

We seek a lower bound for 
$$\inf \int_{X} \rho^2 \ d\Hdim^2,$$
where the infimum is taken over all Borel measurable functions $\rho\colon X \to [0, \infty]$ such that for all $t \in \mathcal{G}$
\begin{equation}\label{modulus assumption} \int_{\gamma_t}\rho \ ds \geq 1.\end{equation}

We claim that the following weighted co-area inequality holds:
\begin{equation}\label{weighted co-area} \int_{\mathcal{G}}\int_{\gamma_t} \rho^2 \ d\Hdim^1d\Hdim^1 \leq \omega \int_X \rho^2 \ d\Hdim^2.\end{equation}
If $\rho^2$ is the characteristic function of a $\Hdim^2$-measurable subset $A \subeq X$, then \eqref{weighted co-area} follows from the usual co-area inequality given in Theorem \ref{co-area} applied to the metric space $X \cap A$.  Linearity of the integral then shows that \eqref{weighted co-area} holds if $\rho^2$ is a simple function, and the standard limiting argument using the monotone convergence theorem shows that \eqref{weighted co-area} holds in the desired generality.

Thus, applying \eqref{weighted co-area} and H\"older's inequality,
$$\int_{X} \rho^2 \ d\Hdim^2 \geq \omega\inv \int_{\mathcal{G}}\frac{1}{\Hdim^1(\gamma(t))}\left(\int_{\gamma_t}\rho \ d\Hdim^1\right)^2 \ d\Hdim^1.$$
Since $\gamma_t$ is an arc, 
$$\int_{\gamma_t} \rho \ d\Hdim^1 =  \int_{\gamma_t} \rho \ ds \geq 1,$$
and so by the definition of $\mathcal{G}$ and the measure estimate for $\mathcal{G}$, we have
$$\int_{X} \rho^2 \ d\Hdim^2 \geq \omega\inv \int_{\mathcal{G}}\frac{1}{8C\omega N^2s_0r} \ d\Hdim^1 \geq \frac{1}{16C\omega^2N^2s_0^2} =: L.$$
\end{proof}

\begin{remark} Arguing as in \cite[Proposition 3.1]{ConfDim}, Corollary \ref{modulus cor} implies that a locally Ahlfors $2$-regular and $LLLC$ metric space $(X,d)$ homeomorphic to a surface is locally $2$-Loewner.  As in \cite[Section 5]{Acta}, this implies that $(X,d)$ locally supports a weak $(1,2)$-Poincar\'e inequality.  Using the techniques of \cite[Theorem 5.9]{Acta}, Corollary \ref{modulus cor} could probably be improved to show that $(X,d)$ locally supports a weak $(1,1)$-Poincar\'e inequality, which would be the full local analog of Semmes result \cite[Theorem B.10]{QuantTop} in dimension $2$.
\end{remark}

\subsection{Construction of quasicircles}
\indent

In this section, we construct chord-arc loops (and hence quasicircles) at specified scales under quite general hypotheses.  

\begin{theorem}\label{quasicircle}
Let $(X,d)$ be a proper metric space, and $z \in X$.  Suppose that $U \subeq X$ is a neighborhood of $z$ and that there is a radius $R_0>0$ and constants $\Lambda, D, L \geq 1$ and $Q \geq 0$ such that  
\begin{itemize}
\item[(i)] $U$ is homeomorphic to the plane $\reals^2$,
\item[(ii)] $B(z, R_0) \subeq U$ and $R_0 \leq \diam(U)$,
\item[(iii)] $U$ has  relative Assouad dimension at most $Q$, with constant $D$,
\item[(iv)] $U$ is relatively $\Lambda$-linearly locally contractible, 
\item[(v)] each pair of points $x, y \in B(z, R_0)$ may be connected by a path of length at most $L d(x,y)$.
\end{itemize}
Then there exist constants $\lambda, C_1, C_2 \geq 1$, depending only on $\Lambda,D,Q$ and $L$, such that if  $R \leq R_0/C_1$, there is a $\lambda$-chord-arc loop $\gamma\colon \sphere^1 \to B(z, R_0)$ with $\ind(\gamma, z)\neq 0$ and 
\begin{equation}\label{gamma inclusions} \frac{R}{C_2} \leq \dist(z, \im(\gamma)) \leq C_2R, \mand \frac{R}{C_2} \leq \diam(\im(\gamma)) \leq C_2R.\end{equation}
\end{theorem}
 
The rest of this section will consist of the proof of this theorem, and so throughout we let $(X,d)$, $U$, $z$, etc., be as in the hypotheses of Theorem \ref{quasicircle}.  The construction has two steps.  First, we create a polygon with a controlled number of vertices surrounding $z$ at the correct scale.  We then minimize a functional on loops surrounding $z$.  The minimizer will be a chord-arc loop.

For $a, b \in X$, define 
$$d'(a,b)=\inf \lng_{d}(\gamma),$$
where the infimum is taken over all paths $\gamma \colon [0,1] \to X$ with $\gamma(0)=a$ and $\gamma(1)=b$.  Assumption $(v)$ implies that $d'$ is a metric on $B_d(z, R_0)$, and shows that if $a, b \in B_d(z, R_0)$, then
\begin{equation}\label{path metric ineq}
d(a,b) \leq d'(a,b) \leq Ld(a,b).
\end{equation}
Note that the first inequality in \eqref{path metric ineq} is valid for all points $a, b \in X$.  Thus if $x \in B_d(z, R_0/2)$, and $0<r \leq R_0/2$, we have 
\begin{equation}\label{path metric inclusions}
B_d\left(x , \frac{r}{L} \right) \subeq B_{d'}(x, r) \subeq B_{d}(x, r) \subeq B_d(z, R_0).
\end{equation}

As $(X,d)$ is proper, the ball $\ovl{B}_d(z, 2LR_0)$ is compact.  Thus for any pair of points $x, y \in B_d(z, R_0)$, there is a (not necessarily unique) path in $X$ whose $d$-length is $d'(x,y)$ \cite[2.5.19]{Burago}.  The following lemma shows that such a path is a geodesic in the $d'$-metric.  

\begin{lemma}\label{d vs d'} Let $\gamma\colon [0,1] \to X$ be a path such that 
$$\lng_{d}(\gamma) = d'(\gamma(0), \gamma(1)).$$
Then $\lng_{d'}(\gamma) = d'(\gamma(0), \gamma(1))$ as well.
\end{lemma}

\begin{proof} We first claim that for all $s, t \in [0,1]$, we have 
$$\lng_d(\gamma|_{[s,t]}) = d'(\gamma(s), \gamma(t)).$$
The definition of $d'$ shows that 
$$\lng_d(\gamma|_{[s,t]}) \geq d'(\gamma(s), \gamma(t)).$$
If this inequality is is strict, then there is a path $\beta$ in $X$ connecting $\gamma(s)$ to $\gamma(t)$ with 
$$\lng_d(\beta) < \lng_d(\gamma|_{[s,t]}).$$
This implies that 
$$\lng_d(\gamma|_{[0,s]} \cdot \beta \cdot \gamma|_{[t,1]}) < \lng_d(\gamma) = d'(\gamma(0), \gamma(1)).$$
This is a contradiction. 

We now prove the lemma. Suppose that $0 = t_1 < \hdots < t_n =1$ is a partition of $[0,1]$.  Then by the claim,
$$\sum_{i = 1}^{n-1} d'(\gamma(t_i), \gamma(t_{i+1})) = \sum_{i=1}^{n-1} \lng(\gamma|_{[t_i, t_{i+1}]}) = \lng(\gamma).$$
Since this is true for each partition, it is true for the supremum over all partitions.  The lemma follows.
\end{proof}

We denote the image of a $d'$-geodesic connecting points $x$ and $y$ by $[x, y]$, following the conventions for geodesics laid out in Section \ref{defs}.

Let $$B_0 = \ovl{B}_{d'}\left(z, \frac{R_0}{16\Lambda L}\right);$$
our construction will take place inside this set.  Note that by \eqref{path metric inclusions}, $B_0 \subeq B_d(z, R_0/2) \subeq U$.  Since $U$ is connected, it follows that 
$$\frac{R_0}{16\Lambda L} \leq \diam_{d'}(B_0) \frac{R_0}{8 \Lambda L}.$$

\begin{lemma}\label{relative properties} In the $d'$-metric, the set $B_0$ is relatively $\Lambda L$-linearly locally contractible, and has relative Assouad dimension at most $Q$ with constant $DL^Q$. 
\end{lemma}
\begin{proof} Let $x \in B_0$, and let $r \leq \diam_{d'}(B_0) \leq R_0/8\Lambda L$. In particular, this implies that $x \in U$ and $r\leq \diam_{d}(U)$. 

We first show that $B_{d'}(x, r)$ contracts inside $B_{d'}(x, {\Lambda} Lr).$  Since $U$ is relatively ${\Lambda} $-linearly locally contractible, the ball $B_d(x, r)$ contracts inside $B_d(x, {\Lambda} r)$.  Since ${\Lambda} Lr \leq R_0/2$, we may apply \eqref{path metric inclusions} to see that $B_d(x, {\Lambda} r) \subeq B_{d'}(x, {\Lambda} Lr).$ Thus $B_{d}(x, r)$ contracts inside $B_{d'}(x, {\Lambda} Lr).$ Since $B_{d'}(x,r) \subeq B_d(x,r)$, this suffices.

Let $0 <\ep \leq 1/2$; we now show that the ball $B_{d'}(x, r)$ can be covered by a controlled number of $d'$-balls of radius $\ep r.$  Since $U$ is relatively doubling of dimension $Q$ with constant $D$, the ball $B_d(x,r)$ may be covered by at most $DL^Q\ep^{-Q}$ balls $\{B_i = B_d(x_i, \ep r/L)\}_{i \in I}$.  We may assume that for each $i \in I$, $d(x_i, x) \leq r+\ep r/L$, for otherwise $B_i \cap B_d(x, r) = \emptyset.$  Then for each $i \in I$,
$$d(x_i,z) \leq d(x_i, x) + d(x, z) \leq 2r + \frac{R_0}{16{\Lambda} L} \leq \frac{R_0}{2}.$$ 
Since $\ep r \leq R_0/2$, the inclusions \eqref{path metric inclusions} imply that $B_i \subeq B_{d'}(x_i, \ep r)$ for each $i \in I$.  Since $B_{d'}(x, r) \subeq B_{d}(x,r),$ this completes the proof. 
\end{proof}

The assumption of linear local connectivity implies that a loop which stays far away from a given point either has large diameter or has index zero with respect to that point.

\begin{lemma}\label{dist diam}  Let $a >0$, and let $0 < R \leq R_0/16a$.  Suppose that $\alpha \colon \sphere^1 \to B_0$ is a continuous map with $\ind(\alpha, z) \neq 0$ and $\dist_{d'}(z, \im(\alpha))\geq aR.$  Then $\diam_{d'}(\im \alpha) \geq aR/{\Lambda} L.$
\end{lemma}

\begin{proof}
 Suppose that $\diam_{d'}(\im \alpha) < aR/{\Lambda} L.$  Let $x \in \im\alpha$; then we have $\im \alpha \subeq B_{d'}(x, aR/{\Lambda} L).$  Then $x \in B_0$, and the upper bound on $R$ implies that $aR/{\Lambda} L \leq \diam_{d'}(B_0)$.  By Lemma \ref{relative properties}, $B_0$ is relatively ${\Lambda} L$-linearly locally contractible, and so $\alpha$ is homotopic to a point inside of $B_{d'}(x, aR).$ But by assumption $d'(z, x) \geq aR$, which implies that the tracks of the homotopy do not meet $z$.  This is a contradiction with the fact that $\ind(\alpha, z)\neq 0$.  
\end{proof}

We now begin the construction of the polygon discussed above.

\begin{lemma}\label{first polygon}  Let $0 < R \leq R_0/48{\Lambda} L.$ Then there exists a path $\beta \colon \sphere^1 \to B_0$ and a constant $C_0$ such that the following statements hold:

\begin{itemize}
\item[(i)] $\ind(\beta, z) \neq 0.$
\item[(ii)] $R/2 \leq \dist_{d'}(z,\im{\beta}) \leq 3R,$ 
\item[(iii)] $\lng_{d'}(\beta) \leq C_0R.$
\end{itemize}
The constant $C_0$ depends only on $D, Q, {\Lambda} ,$ and $L$.
\end{lemma} 

\begin{proof}
Set
$$\ep = \frac{R}{128{\Lambda} L({\Lambda} L+1)}.$$
The number $\ep$ will be roughly the distance between vertices of the polygon to be constructed.  By the definition of the $d'$-metric, the ball $B_{d'}(z, R)$ is connected and has compact closure in $U$, so by Lemma \ref{Jordan curve approximation} there is an embedding $\alpha \colon \sphere^1 \to B_{d'}(z, 2R)$ with $B_{d'}(z, R) \subeq \ins(\im{\alpha}).$ Thus 
\begin{equation}\label{alpha dist}
R \leq \dist_{d'}(z, \im \alpha) \leq 2R \mand \ind(\alpha, z) \neq 0.
\end{equation}
This implies in particular, that $\im{\alpha} \subeq B_0$. 
Let $S$ be a maximal $\ep$-separated set in $\im\alpha$, with respect to the $d'$ metric. By Lemma \ref{relative properties}, $B_0$ is relatively doubling of dimension $Q$ with constant $DL^Q$.  Since $S \subeq \im\alpha \subeq B_{d'}(z, 2R),$ and $2R \leq \diam_{d'}(B_0),$ we see from the definition of $\ep$ that $\card{S}$ is bounded above by a number that depends only on $D, Q, {\Lambda} $ and $L$. 

By Proposition \ref{jordan curve segments}, we may find a finite and cyclically ordered set of points $\{\psi_1, \hdots, \psi_n\} \subeq \sphere^1$ such that for $i=1,\hdots, n$, $\modu n,$
\begin{equation}\label{w choice}
\diam_{d'}(\alpha([\psi_i, \psi_{i+1}])) < \ep.
\end{equation}
Note that we have no control over the size of $n$.  

We inductively define a sequence of indices $\{i_k\} \subeq \{1,\hdots, n\}$ and a sequence of points $\{z_k\} \subeq S$ as follows.  Let $i_i = 1$, and let $z_1 \in S$ be any point such that $d'(z_1, \alpha(\psi_{i_1})) <\ep.$  Such a point exists since $S$ was chosen to be maximal.   Now suppose that $z_k \in S$ and $i_k\in \{1, \hdots, n\}$ have been chosen.  Let $i_{k+1}$ be the smallest index $j$ greater than ${i_k}$ such that $\alpha(\psi_{j}) \notin B_{d'}(z_{k}, \ep).$ If no such index exists, the process stops.   If $i_{k+1}$ may be found, set $z_{k+1}$ to be any point in $S$ such that $d'(z_{k+1}, \alpha(\psi_{i_{k+1}})) < \ep$.  Since $i_k < i_{k+1} \leq n$, this process stops after finitely many iterations.  Let $z_m$ and $i_{m}$ be  the final point and index produced.

\begin{figure}[h]\label{induct}
\begin{center}
\psfrag{apik}{$\alpha(\psi_{i_k})$}
\psfrag{apik1}{$\alpha(\psi_{i_{k+1}})$}
\psfrag{zk}{$z_k$}
\psfrag{e}{$\ep$}
\includegraphics[width=.75\textwidth]{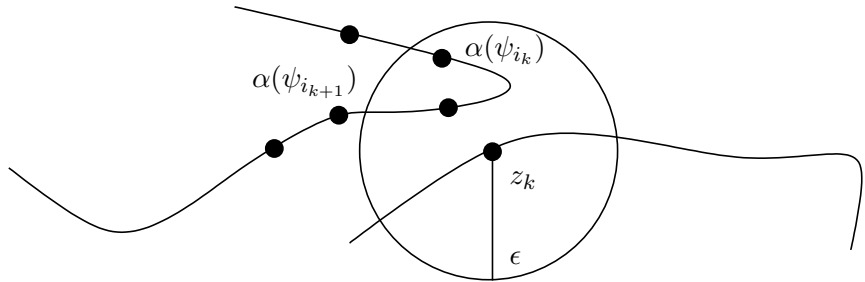}
\caption{Choosing $z_k$}
\end{center}
\end{figure}

The result of this process satisfies 
\begin{itemize}
\item[(a)] $d'(\alpha(\psi_i), z_k) < \ep$ if $1 \leq k < m$ and $ i_{k} \leq i <{i_{k+1}}$,
\item[(b)] $d'(\alpha(\psi_i), z_m) < \ep$ if $i_{m} \leq i \leq n$.  
\end{itemize}
From (a) and \eqref{w choice}, we see that if $1 \leq k < m$,
$$d'(z_k, z_{k+1}) \leq d'(z_k, \alpha(\psi_{i_{k+1}-1})) + d'(\alpha(\psi_{i_{k+1}-1}), \alpha(\psi_{i_{k+1}})) + d'(\alpha(\psi_{i_{k+1}}), z_{k+1}) < 3\ep,$$
and (b) shows that 
$$d'(z_m, z_1) \leq d'(z_m, \alpha(\psi_n)) + d'(\alpha(\psi_n), \alpha(\psi_1)) + d'(\alpha(\psi_1), z_1) < 3\ep.$$
These inequalities imply that 
\begin{equation}\label{geo diam}\lng_{d'}([z_k, z_{k+1}]) =\diam_{d'}([z_k, z_{k+1}]) < 3\ep
\end{equation} for $k=1, \hdots, m$, $\modu m.$  

Furthermore, (a) and (b) together with \eqref{w choice} show that for $k=1, \hdots, m$, $\modu m$,
$$\alpha([\psi_{i_k} ,\psi_{i_{k+1}}]) \subeq B_{d'}(z_k, 2\ep),$$
and so for $k=1, \hdots, m$, $\modu m$, 
\begin{equation}\label{alpha diam}
\diam_{d'}(\alpha([\psi_{i_k}, \psi_{i_{k+1}}])) \leq 4\ep,
\end{equation}

Rename $\psi_{i_k}=\theta_k$ for $k=1, \hdots, m,$ and let $\beta \colon \sphere^1 \to X$ be defined by 
$$ \beta(\theta) = s^{[z_k, z_{k+1}]}_{[\theta_{k}, \theta_{k+1}]}(\theta), \ \theta \in [\theta_{k}, \theta_{k+1}].$$  We wish to show that $\alpha$ and $\beta$ are homotopic away from $z$. To do so, we apply Lemma \ref{homotopy trick} to $\alpha$ and $\beta$, using $\del = 4\ep$.  We now verify the hypotheses of Lemma \ref{homotopy trick}.  Recall that by Lemma \ref{relative properties}, $B_0$ is relatively linearly locally contractible in the $d'$-metric, and by the connectedness of $U$, $R_0/16{\Lambda} L \leq \diam_{d'} B_0.$   If $x \in B_0$ and $0 < r \leq 8\ep({\Lambda} L+ 1)$, the definition of $\ep$ and the assumption that $R \leq R_0$ show that $r \leq \diam_{d'}B_0$, and so the ball $B_{d'}(x, r)$ contracts inside $B_{d'}(x, {\Lambda} Lr).$  The inequalities (a), (b), \eqref{geo diam}, and \eqref{alpha diam} show that the remaining hypotheses of Lemma \ref{homotopy trick} are fulfilled, showing that $\beta$ is homotopic to $\alpha$ inside the $8{\Lambda} L({\Lambda} L +1)\ep$-neighborhood of $\alpha$ (in the $d'$-metric).  Since $8{\Lambda} L({\Lambda} L + 1)\ep= R/16$ and $\dist_{d'}(z, \im\alpha)\geq R$, we see that the tracks of the homotopy do not meet $z$.  Conclusions $(i)$ and $(ii)$ now follow from \eqref{alpha dist}.   

It could be the case that $z_k = z_l$ for $k \neq l$.  As a result, we may not conclude that $m \leq \card{S}$.  However, we may decompose $z_1, \hdots, z_m$ in to a finite collection of cycles where no $z_i$ is repeated.  The map $\beta$ can then be considered as the concatenation of the restrictions to corresponding parameter segments.  Since $\ind(\beta, z) \neq 0$, at least one such restriction must also have non-zero index.  As the resulting loop is a subset of $\im{\beta}$, conclusions $(i)$ and $(ii)$ persist.   Repeating this procedure finitely many times, we may assume without loss of generality that $z_1, \hdots, z_m \subeq S$ are distinct points.  As the cardinality of $S$ depends only on $D, Q, {\Lambda} $, and $c$, condition $(iii)$ now follows from \eqref{geo diam}. Note that by conclusion (ii) and the assumption that $R \leq R_0/48\Lambda L$, the image of $\beta$ is contained in $B_0$. \end{proof}

To complete the proof of Theorem \ref{quasicircle}, we need the following technical fact regarding lower semi-continuity of path integrals.  

\begin{lemma}\label{semi-continuity}  Let $\rho: X \to [0,\infty)$ be a lower semi-continuous function on a metric space $(X,d)$, and suppose that $\{\gamma_n\}_{n \in \nats}$ is a sequence of loops in $X$ of uniformly bounded length.  If $\gamma_n$ converges uniformly to a loop $\gamma$ in $X$, then 
$$\int_{\gamma} \rho \ ds \leq \liminf_{n \to \infty} \int_{\gamma_n} \rho \ ds.$$
\end{lemma}

\begin{proof}  The proof of this fact is given in the last three paragraphs of the proof of \cite[Prop. 2.17]{Acta}.  The key fact is that a $1$-Lipschitz function $s: I \to \reals$, where $I$ is any interval, is differentiable almost everywhere and satisfies $s'(t) \leq 1$ almost everywhere.
\end{proof}

\begin{proof}[Proof of Theorem \ref{quasicircle}]  
Let $C_0$ be the constant provided by Lemma \ref{first polygon}, and set
$$C_1 = 320C_0({\Lambda} L+2).$$
Fix $0 < R \leq R_0/C_1$.  

Define the continuous function $\rho \colon B_0\bslash \{z\} \to [0, \infty)$ by 
$$\rho(x)=\left(\frac{R}{\dist_{d'}(z, x)}\right)^2 + 1,$$
and for any rectifiable loop $\gamma \colon \sphere^1 \to B_0$, define the functional 
$$\sigma(\gamma) = \int_{\gamma}\rho \ ds.$$
The function $\sigma$ balances the length of a loop against its distance to $z$. 

If $\beta$ is the loop given by Lemma \ref{first polygon}, we have 
$$0< \sigma(\beta) \leq 5\lng_{d'}(\beta) \leq 5C_0 R.$$
For $n\in \nats$, we may find rectifiable loops $\gamma_n \colon \sphere^1 \to B_0\bslash \{z\}$, such that $\ind(\gamma_n, z) \neq 0$ and 
$$ \lim_{n \to \infty} \sigma(\gamma_n) = \inf \sigma(\gamma),$$
where the infimum is taken over all rectifiable loops $\gamma$ in $B_0\bslash \{z\}$ with $\ind(\gamma, z) \neq 0$.  Without loss of generality, we may assume that $\sigma(\gamma_n) < 2\sigma(\beta) \leq 10C_0R,$ for all $n \in \nats$.  

Fix $n \in \nats$. Our first task is to show that the loop $\gamma_n$ lies in a controlled annulus around $z$.  Define 
$$d_n = \min\{\rho \circ \gamma_n(\theta) : \theta \in \sphere^1\} \mand D_n = \max\{\rho\circ \gamma_n(\theta) : \theta \in \sphere^1\}, $$
and set $l_n = \lng_{d'}(\gamma_n).$  
We have that 
\begin{equation}\label{gamma_n  length} l_n \leq \sigma(\gamma_n) <10C_0R.\end{equation}

Suppose that $d_n \geq 2{\Lambda} Ll_n.$  If $x \in \im(\gamma_n) \subeq B_0$, we have that 
$$\im(\gamma_n) \subeq B(x, 2l_n).$$
Recall that $B_0$ is relatively ${\Lambda} L$-linearly locally contractible by Lemma \ref{relative properties}.  From  \eqref{gamma_n length} and the definition of $C_1$, we see that $2l_n \leq \diam_{d'}B_0.$  Thus $\gamma_n$ is homotopic to a point inside of $B(x, 2{\Lambda} Ll_n) \subeq B_0 \subeq U$.  However, $d(x, z) \geq d_n \geq 2\Lambda Ll_n$, and so we have a contradiction with the assumption that $\ind(\gamma_n, z) \neq 0$.  Thus 
\begin{equation}\label{d_n upper re length} d_n < 2{\Lambda} Ll_n. \end{equation}
The estimates \eqref{gamma_n length} and \eqref{d_n upper re length} now imply that
\begin{equation}\label{D_n bound} D_n \leq d_n + l_n \leq  10C_0(2{\Lambda} L+1)R.\end{equation}

We now derive a lower bound for $d_n$.  We do so in two cases; first assume that $D_n \leq 4d_n$.  Using \eqref{d_n upper re length}, we have 
$$10C_0 R \geq \sigma(\gamma_n) \geq \int_{\gamma_n} \left(\frac{R}{4d_n}\right)^2 \ ds \geq \left(\frac{R}{4d_n}\right)^2 l_n \geq \frac{R^2}{32{\Lambda} L d_n} .$$
Thus 
\begin{equation}\label{d_n lower} d_n \geq \frac{R}{320C_0{\Lambda} L}.\end{equation}
Now assume that $D_n > 4d_n.$  The triangle inequality shows  
$$10C_0R \geq \sigma(\gamma_n) \geq \int_{d_n}^{D_n} \left(\frac{R}{t}\right)^2\ dt = R^2\left(\frac{1}{d_n} - \frac{1}{D_n}\right) \geq \frac{3R^2}{4d_n},$$
which yields $d_n \geq 3R/40C_0$.  In either case \eqref{d_n lower} holds.  

The compactness of $\sphere^1$ and the length bound \eqref{gamma_n length} imply that the family $\{\gamma_n\}_{n \in \nats}$ is equicontinuous.  By the Arzela-Ascoli Theorem, after passing to a subsequence, the loops $\gamma_n$ converge uniformly to a loop $\gamma_0 \colon \sphere^1 \to B_0$ such that the following hold:
\begin{equation}\label{gamma_0 length} \lng(\gamma_0) \leq 10C_0R, \end{equation}
\begin{equation}\label{d_0 lower}d_0 :=  \min\{d_{\gamma_0}(\theta) : \theta \in \sphere^1\} \geq \frac{R}{320C_0{\Lambda} L},\end{equation}
\begin{equation}\label{D_0 bound} D_0:= \max\{d_{\gamma_0}(\theta) : \theta \in \sphere^1\} \leq 10C_0(2{\Lambda} L+1)R. \end{equation}
 By Lemma \ref{semi-continuity}, $\gamma_0$ minimizes the functional $\sigma$ over all rectifiable loops in $B_0$ with non-zero index with respect to $z$.  From \eqref{d_n lower} and \eqref{d_0 lower}, we see that for sufficiently large $n \in \nats$, the loop $\gamma_n$ is homotopic to $\gamma_0$ without hitting $z$.  Thus $\ind(\gamma_0, z) \neq 0$. Furthermore, the estimate \eqref{d_0 lower} shows that we may apply Lemma \ref{dist diam} to $\gamma_0$ with $a=(320C_0{\Lambda} L)\inv$, concluding that 
 $$\diam_{d'}(\im(\gamma_0)) \geq \frac{R}{320C_0({\Lambda} L)^2}.$$
 From these facts and the comparability of the metrics $d$ and $d'$ given in \eqref{path metric ineq}, we conclude that there is a constant $C_2$, depending only on ${\Lambda} , D, Q$ and $L$, such that $\gamma_0$ satisfies the conditions in \eqref{gamma inclusions}.

It remains to show that $\gamma_0$ is a chord-arc circle with an appropriate constant. Since $\gamma_0 \subeq B_0$, the comparability of $d'$ and $d$ stated in \eqref{path metric ineq} shows that it suffices to check the chord-arc condition in the $d'$ metric.  Let $\phi, \psi \in \sphere^1$, and let $J_1$ and $J_2$ denote the subarcs of $\sphere^1$ whose union is $\sphere^1$ and whose intersection is $\{\phi, \psi\}$.  We will show that 
\begin{equation}\label{not chord-arc}\min\{\lng_{d'}(\gamma_0|_{J_1}),\lng_{d'}(\gamma_0|_{J_2})\}  \leq  (6400C_0^2{\Lambda} L({\Lambda} L+1))^2d'(\gamma_0(\phi), \gamma_0(\psi)).\end{equation}
We first assume that 
\begin{equation}\label{small scale} d'(\gamma_0(\phi), \gamma_0(\psi)) \leq  \frac{R}{640C_0{\Lambda} L}. \end{equation}
 By \eqref{D_0 bound} and the definition of $C_1$, the geodesic segment $[\gamma_0(\phi), \gamma_0(\psi)]$ is contained in $B_0$.  By \eqref{d_0 lower}, it does not meet $z$.  By the additivity of index under concatenation, we may assume without loss of generality that the loop 
$$\til{\gamma}_0(\theta) = \begin{cases}
					\gamma_0(\theta) & \theta \in J_1\\
					s_{J_2}^{[\gamma_0(\phi), \gamma_0(\psi)]}(\theta) & \theta \in J_2 \\
					\end{cases}$$
satisfies $\ind(\til{\gamma}_0, z) \neq 0.$ 
 
\begin{figure}[h]\label{pick}
\begin{center}
\psfrag{gJ1}{$\gamma_0(J_1)$}
\psfrag{gJ2}{$\gamma_0(J_2)$}
\psfrag{z}{$z$}
\psfrag{s}{$s_{J_2}^{[\gamma_0(\phi), \gamma_0(\psi)]}(J_2)$}
\includegraphics[width=.75 \textwidth]{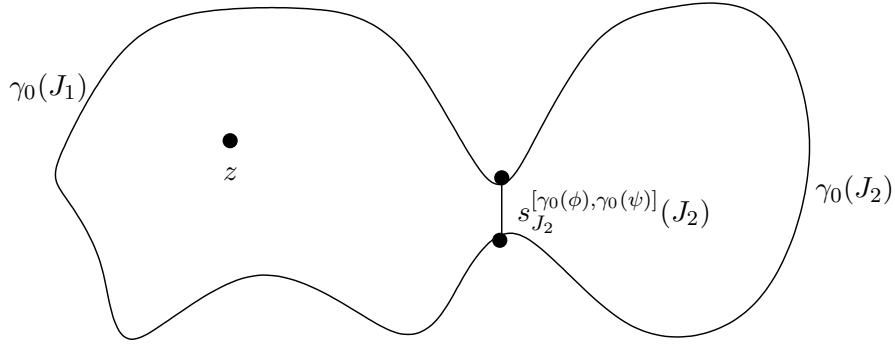}
\caption{A shortcut}
\end{center}
\end{figure}

 Note that for $\theta \in J_2$, \eqref{d_0 lower} and \eqref{small scale} imply   
$$d_{\til{\gamma}_0}(\theta) \geq \frac{R}{640C_0{\Lambda} L}.$$
Using this, \eqref{D_0 bound}, and the fact that $\gamma_0$ and $\til{\gamma}_0$ agree on $J_1$, we see that 
\begin{align*} \sigma(\gamma_0) - \sigma(\til{\gamma}_0)  \geq 
			&  \left(\frac{1}{(10C_0(2{\Lambda} L+1)^2} + 1\right)\lng_{d'}(\gamma_0|_{J_2}) -  \\ &\left((640C_0{\Lambda} L)^2 +1\right) \lng_{d'}(\til{\gamma}_0|_{J_2}). 
						\end{align*}
Since 
$$\lng(\gamma_0|_{J_2}) \geq \lng(\til{\gamma}_0|_{J_2}) = d'(\gamma_0(\phi), \gamma_0(\psi)),$$ the preceding esitmate implies that 
$$ \sigma(\gamma_0) - \sigma(\til{\gamma}_0)  \geq 
			 \left(\frac{1}{10C_0(2{\Lambda} L+1)}\right)^2\lng_{d'}(\gamma_0|_{J_2}) -  \left(640C_0{\Lambda} L\right)^2 \lng_{d'}(\til{\gamma}_0|_{J_2}). $$
						
If \eqref{not chord-arc} does not hold, this implies that $\sigma(\til{\gamma}_0) < \sigma(\gamma_0)$, a contradiction. 

We now assume
$$ d'(\gamma_0(\phi), \gamma_0(\psi)) \geq  \frac{R}{640C_0{\Lambda} L}. $$
By \eqref{gamma_0 length}, we have 
$$\min\{\lng(\gamma_0|_{J_1}),\lng(\gamma_0|_{J_2})\} \leq \lng(\gamma_0) \leq  (6400C_0^2{\Lambda} L )d'(\gamma_0(\phi), \gamma_0(\psi)).$$
This verifies \eqref{not chord-arc}, showing that $\gamma_0$ is a chord-arc loop with appropriate constant, and completes the proof.
\end{proof}

\subsection{Porosity of quasicircles}
\indent 

We now show, in particular, that a quasicircle which is the metric boundary of a metric disk is porous in the completed space.  For quasicircles in the plane, this result is well known. 
Porous sets are small in a quantitative sense.  We will use this concept to get around the fact that subsets of an Ahlfors regular space need not be Ahlfors regular.

\begin{definition}\label{porosity def} A subset $Y$ of a metric space $(X,d)$ is $C$-porous, $C \geq 1$, if for every $y \in Y$ and $0 < r \leq \diam(X)$, there exists a point $x \in X$ such that 
$$B \left(x, \frac{r}{C}\right) \subeq B(y, r) \bslash Y.$$
\end{definition}

\begin{theorem}\label{porosity thm}  Let $\lambda \geq 1$, and suppose that $(X,d)$ is a metric space homeomorphic to the plane such that $\ovl{X}$ is compact and $\lambda$-$LLC$, and $\partial X$ is a $\lambda$-$LLC$ Jordan curve.  Then $\partial X$ is porous in $\ovl{X}$ with constant depending only on $\lambda$.  
\end{theorem}

We will need a version of Janiszewski's separation theorem.  We show how the variant can be derived from the original, a proof of which may be found in \cite[V.9]{Newman}. 
A subset $A$ of a topological space $X$ is said to separate points $u, v \in X$ if $u$ and $v$ are in different components of $X\bslash A$.  
  
\begin{theorem}[Janiszewski's Theorem]\label{Jan theorem II} Let $A$ and $B$ be disjoint closed subsets of $\reals^2$.  If $u,v \in \reals^2$ are such that neither $A$ nor $B$ separates $u$ from $v$, then $A\cup B$ does not separate $u$ from $v$.  
\end{theorem}

\begin{theorem}[Janiszewski's Theorem in $\ovl{\mathbb{D}}^2$]\label{disk Jan theorem}  Let $A, B \subeq \ovl{\mathbb{D}}^2$ be disjoint continua. If $u, v \in \ovl{\mathbb{D}}^2$ are such that neither $A$ nor $B$ separates $u$ from $v$, then $A \cup B$ does not separate $u$ from $v$.  
\end{theorem}
\begin{proof} Since $A$ does not separate $u$ from $v$, we may find a path $\gamma\colon [0,1] \to \ovl{\mathbb{D}}^2$ such that $\gamma(0)=u$ and $\gamma(1)=v$ and satisfying $A\cap \im{\gamma} = \emptyset$.   Similarly, we have a path $\beta \colon [0,1] \to \ovl{\mathbb{D}}^2$ such that $\beta(0)=u$, $\beta(1)=v$, and $B \cap \im{\beta} = \emptyset.$ Since $A$ and $B$ are compact, we may find $0 < \ep < 1$ such that 
\begin{equation}\label{shrink} \{(1-\del)u : 0 < \del \leq \ep\} \cap (A \cup B) = \emptyset = \{(1-\del)v : 0 < \del \leq \ep\} \cap (A \cup B).\end{equation} 
Furthermore, since $\im{\gamma}$ and $\im{\beta}$ are also compact, we may find $0< \ep' \leq \ep$ so that the paths $\gamma' \colon [0,1] \to \disk$ and $\beta'\colon [0,1] \to \disk$ defined by
$$\gamma'(t) = (1-\ep')\gamma(t) \mand \beta'(t)=(1-\ep')\beta(t)$$
do not intersect $(A \cup B)$.
Let $u' =(1-\ep')u$ and $v'=(1-\ep')v$. Note that $u', v' \in \disk$.  Then $\gamma'$ connects $u'$ to $v'$ without intersecting $A$, and $\beta'$ connects $u'$ to $v'$ outside of $B$.  Since $\disk$ is homeomorphic to $\reals^2$, Theorem \ref{Jan theorem II} provides a continuum $E \subeq \disk$ containing $u'$ and $v'$ such that $E \cap (A \cup B) = \emptyset$.  
Let $\gamma_u \colon [0, \ep'] \to \ovl{\mathbb{D}}^2$ be defined by $\gamma_u(t)= (1-t)u$, and similarly define $\gamma_v$.  Then by \eqref{shrink}, $\im(\gamma_u) \cup E \cup \im(\gamma_v)$ is a continuum connecting $u$ to $v$ which does not intersect $A \cup B$, as desired.
\end{proof}

\begin{proof}[Proof of Theorem \ref{porosity thm}.]
We may assume without loss of generality that the boundary $\partial{X}$ satisfies the so-called three point condition given by \eqref{three point} with constant $\lambda$.  Let $z \in \partial X$, and let $0 \leq r \leq \diam(X)$.  We consider three cases.

\bigskip\noindent\textit{Case 1: $0 \leq r < \diam(\partial X)/4\lambda $.} In this case, we may find a point $w \in \partial X$ such that $d(z, w) \geq 2\lambda r$.  We may also find points $u, v \in \partial{X}$ such that $\{z, u, w, v\}$ is cyclically ordered on $\partial{X}$, $d(z, u)=r/4\lambda=d(z,v)$, and if $J(z)$ is the component of $\partial{X}\bslash \{u,v\}$ which contains $z$, then $J(z) \subeq B(z, r/4\lambda).$  

Let $A$ be the component of $\ovl{B}(z, r/8\lambda^2)$ containing $z$, and let $B$ be the component of $\ovl{X}\bslash B(z, r/2\lambda)$ containing $w$.  Since $\ovl{X}$ is $\lambda$-$LLC$, we see that 
\begin{equation}\label{A B containments} B\left(z, \frac{r}{8\lambda^3}\right) \subeq A, \mand \ovl{X} \bslash B(z, r/2) \subeq B.\end{equation}

As components of compact sets, $A$ and $B$ are continua, and are disjoint by definition.  By definition, $\{u\} \cup J(z) \cup \{v\}$ connects $u$ to $v$ inside $\ovl{B}(z, r/4\lambda) \subeq \ovl{X} \bslash B$.  Furthermore, the $LLC_2$ condition shows that $u$ and $v$ may also be connected in $\ovl{X}\bslash B(z, r/4\lambda^2) \subeq \ovl{X} \bslash A.$  Thus by Theorem \ref{disk Jan theorem}, there is a continuum $\alpha \subeq \ovl{X}\bslash (A \cup B)$ which contains both $u$ and $v$.  By \eqref{A B containments}, we have 
\begin{equation}\label{alpha containments}\alpha \subeq \left(\ovl{X}\bslash B\left(z, \frac{r}{8\lambda^3}\right)\right) \cap B(z, r/2).\end{equation}

\begin{figure}[h]\label{porosity}
\begin{center}
\psfrag{r1}{$r/8\lambda^3$}
\psfrag{r2}{$r/8\lambda^2$}
\psfrag{r3}{$r/4\lambda$}
\psfrag{r4}{$r/2\lambda$}
\psfrag{r5}{$r/2$}
\psfrag{z}{$z$}
\psfrag{w}{$w$}
\psfrag{u}{$u$}
\psfrag{v}{$v$}
\psfrag{a}{$\alpha$}
\psfrag{A}{$A$}
\psfrag{B}{$B$}
\psfrag{Iu}{$I_u$}
\psfrag{Iv}{$I_v$}
\includegraphics[width=.75 \textwidth]{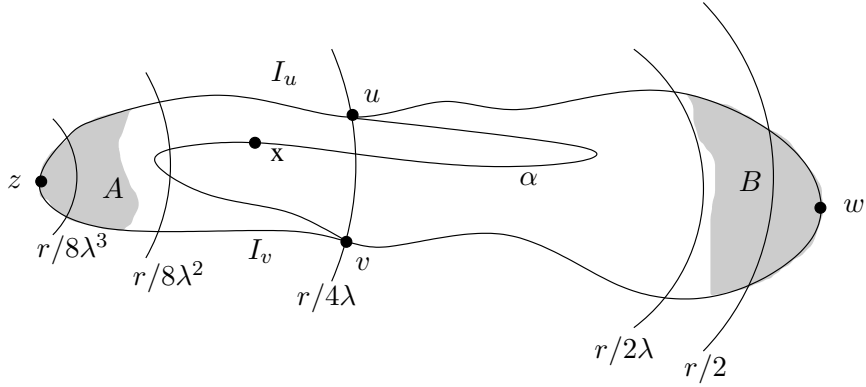}
\caption{After applying Janiszewski's Theorem}
\end{center}
\end{figure}

Let $J(u)$ be the component of $\partial{X} \bslash \{z, w\}$ containing $u$, and set $I(u)= \{z\} \cup J(u) \cup \{w\}.$  Define $I(v)$ similarly. We claim that $\dist(v, I(u)) \geq r/8\lambda^2$.  If this is not the case, the three-point condition implies that either $z$ or $w$ is within distance $r/8\lambda$ of $v$, which is not the case.  Thus by the connectedness of $\alpha$, we may find a point $x \in \alpha$ such that 
$$\dist(x, I(u))=\frac{r}{32\lambda^4}.$$
Suppose there exists a point $y \in I(v)$ such that $d(x,y) < r/32\lambda^4.$  Then $\dist(y, I(u)) \leq r/16\lambda^4$.  Since $\partial{X}$ satisfies the $\lambda$ three-point condition \ref{three point}, this implies that either $z$ or $w$ is contained in $B(y, r/16\lambda^3)$.  However, we have that
$$B\left(y, \frac{r}{16\lambda^3} \right) \subeq B\left(x, \frac{r}{16\lambda^3} + \frac{r}{32\lambda^4}\right) \subeq B\left(x, \frac{r}{8\lambda^3}\right),$$ 
and that $x \in \alpha$.  By \eqref{alpha containments} and the fact that $d(z, w) \geq 2\lambda r$, this is a contradiction.  Thus we see that $\dist(x, I(v)) \geq r/32\lambda^4$.  This along with \eqref{alpha containments} shows that  
$$B\left(x, \frac{r}{32\lambda^4}\right) \subeq B(z, r)\bslash \partial{X}.$$

\bigskip\noindent\textit{Case 2: $8\diam(\partial{X}) \leq r \leq \diam{X}$.}  We may find a point $x \in \ovl{X}$ such that $d(x, z) =r/4.$  Since $z \in \partial{X}$, and $\diam{\partial{X}} \leq r/8$, we see that 
$$\dist(x, \partial{X}) \geq d(x, z) - \diam{\partial{X}} \geq \frac{r}{8}.$$
Thus $B(x, r/8) \subeq B(z,r) \bslash \partial{X}$.  

\bigskip\noindent\textit{Case 3: $\diam(\partial{X})/4\lambda \leq r \leq 8\diam(\partial{X}).$}  We have that 
$$\frac{r}{64\lambda}< \frac{\diam(\partial{X})}{4\lambda},$$
and so by Case 1, there is a point $x \in X$ such that 
$$B\left(x, \frac{r}{2048 \lambda^5} \right) \subeq B\left(z, \frac{r}{64\lambda}\right) \bslash \partial{X} \subeq B(z, r) \bslash \partial{X}.$$

These cases show that $\partial{X}$ is $2048\lambda^5$-porous in $\ovl{X}$.  
\end{proof}

\subsection{The proof of Theorem \ref{existence}}
\indent 

In this subsection, we collect the results proven thus far and complete the proof of Theorem \ref{existence}.   We begin with a theorem that identifies planar sets on a surface, for which we were unable to find a reference.

\begin{proposition}\label{complement prop} Let $\mathcal{S}$ be a surface, and let $U \subeq \mathcal{S}$ be a connected, relatively compact, non-compact, non-empty, open subset such that $\mathcal{S}\bslash U$ is connected and the homomorphism $i_*\colon \pi_{1}(U) \to \pi_1(\mathcal{S})$ induced by the inclusion $i\colon U \to \mathcal{S}$ is trivial. Then $U$ is homeomorphic to the plane. 
\end{proposition}
 
\begin{proof}  It suffices to show that $U$ is simply connected, as follows.  If $U$ is simply connected, then it is orientable \cite[6.2.10]{Spanier}.  Every connected, orientable surface has a Riemann surface structure \cite[II.1.5E]{Sario} .  Since $U$ is non-compact, and non-empty the Uniformization Theorem implies that $U$ is homeomorphic to the plane. 

We first consider the case that $\mathcal{S}$ is not compact.  Suppose that $U$ is not simply connected.  Any continuous loop $\gamma\colon \sphere^1 \to \mathcal{S}$ is homotopic to a loop with only transversal self-intersections \cite[Ch. 2 Sec. 3]{G+P}.  Thus we may find a loop in $U$ with only finitely many self-intersections which represents a non-trivial homotopy class.  By decomposing this loop, we may find a Jordan curve $J$ in $U$ which represents a non-trivial homotopy class. 

We now claim that there is an embedding $h\colon \disk \to \mathcal{S}$ such that the topological boundary of $h(\disk)$ is $J$. We give only a sketch of the proof of this claim, and refer to \cite{Scott} for the details.  By the Uniformization Theorem, the universal cover $\til{\mathcal{S}}$ of $\mathcal{S}$ is homeomorphic to the plane or the sphere.  Since $J$ is null-homotopic, the pre-image of $J$ under the universal covering map is a collection of disjoint Jordan curves.  By Schoenflies' theorem, each such curve is the boundary of an embedded disk $D$ in $\ovl{\mathcal{S}}$.  The group of deck transformations acts fix point free and moves each pre-image of $J$ off of itself.   Again by Schoenflies' theorem, we see that if $g$ is a deck transformation with $g(D) \cap D \neq \emptyset$, then either $g(D) \subeq D$ or $g\inv (D) \subeq D$.  In either case, the Brouwer fixed-point  theorem yields a contradiction. Thus the covering projection restricted to $D$ is a homeomorphism, proving the claim.  

As $J$ represents a non-trivial homotopy class in $U$, we must have $h(\disk) \cap (\mathcal{S} \bslash U) \neq \emptyset$.  As $J \subeq U$,  the intersection $h(\disk) \cap (\mathcal{S} \bslash U)$ is a relatively open and closed subset of $\mathcal{S} \bslash U$, and hence it is all of $\mathcal{S} \bslash U$.  Since $U$ is relatively compact, this implies that $\mathcal{S}$ is compact, a contradiction. 

We now assume that $\mathcal{S}$ is compact. We will make a homological argument; all homology groups will be singular and have coefficients in $\ints_2$.  The reason for this is that any manifold has a unique orientation over $\ints_2$  \cite[6.2.9]{Spanier}, and we will eventually use a duality theorem that requires an orientation.  
  The long exact sequence for homology of the pair $(\mathcal{S}, U)$ includes 
$$\hdots \to H_2(U) \to H_2(\mathcal{S}) \to H_2(\mathcal{S}, U) \to H_1(U) \to H_1(\mathcal{S}) \to \hdots.$$ 
By a version of Alexander duality \cite[6.2.17 and 6.9.9]{Spanier}, there is a natural isomorphism $$H_2(\mathcal{S}, U) \cong \textit{\v{H}}^{0}(\mathcal{S}\bslash U),$$ where $\textit{\v{H}}^\ast$ denotes C\v{e}ch cohomology with coefficients in $\ints_2$.  Since $\mathcal{S}\bslash U$ is connected, we see that $H_2(\mathcal{S}, U) \cong \ints_2$.  As $U$ is open and non-compact, $H_2(U) = 0$ (see \cite{Scott}). By exactness, 
\begin{equation}\label{second hom map} H_2(\mathcal{S}) \to H_2(\mathcal{S},U)\end{equation}
is injective.  On the other hand, $\mathcal{S}$ is orientable over $\ints_2$ and compact, and so  $H_2(\mathcal{S})\cong \ints_2$.  Thus the homomorphism \eqref{second hom map} is surjective as well. Exactness yields that 
\begin{equation}\label{third hom map} H_2(\mathcal{S}, U) \to H_1(U)\end{equation}
is trivial.

Since $i_*\colon \pi_1(U) \to \pi_1(\mathcal{S})$ is trivial, the homomorphism 
$$ H_1(U) \to H_1(\mathcal{S})$$
is trivial as well.  Exactness implies that \eqref{third hom map} is surjective, and hence that $H_1(U)$ vanishes. Since $U$ is open and non-compact, $\pi_1(U)$ is a free group \cite[I.44]{Sario}.  Thus if $\pi_1(U)$ is non-trivial, then $H_1(U; \ints)$ is a free abelian group of positive rank.  By the Universal Coefficient Theorem for homology \cite[Theorem 15.4(a)]{Bott}, $H_1(U)$ contains a subgroup isomorphic to $H_1(U; \ints) \otimes \ints_2$, which is non-trivial.  This is a contradiction, and so we conclude that $\pi_1(U)$ is trivial, as desired.
\end{proof}

We now show that in a proper, $LLLC$ metric surface, the radius associated with the contractibility condition determines the size of planar sets in the surface. 

\begin{proposition}\label{disk size} Let $(X,d)$ be a connected, proper, and linearly locally contractible metric space homeomorphic to a surface, $K \subeq X$ a compact set, and $R_K$ and $\Lambda_K$ be the radius and constant associated to this set by the linear local contractibility condition.   Then for each $z \in K$ and $R \leq R_K/8\Lambda_K$ such that $B(z, 4\Lambda_K R) \subeq K$, there exists a neighborhood $U$ of $z$ which is homeomorphic to the plane and satisfies 
\begin{equation}\label{disk inclusions} B\left(z, \frac{R}{2\Lambda_K}\right) \subeq U \subeq B(z, R).\end{equation}
\end{proposition}

\begin{proof} Let $(X,d)$, $K$, $R_K,$ and $\Lambda_K$ be as in the statement, and fix $z \in K$ and $R \leq R_K/8\Lambda_K.$  

If $X \bslash B(z, R)$ is empty, then $X$ is compact because it is proper. Moreover, the definition of $R$ shows that in this case, $X$ is contractible.  This is a contradiction, as there are no compact, contractible surfaces.  Thus we may assume $X \bslash B(z, R) \neq \emptyset$.  

We first claim that there is a unique component of $X \bslash B(z, R/2\Lambda_K)$ that intersects $X\bslash B(z, R).$ By assumption, the ball $B(z, 4\Lambda_K R)$ is contained in $K$ and its diameter is no greater than $R_K$.  Thus by proposition \ref{LLLC implication}, $B(z, R/2\lambda_K)$ satisfies the first relative $LLC$ condition with constant $\Lambda_K$.  Let $V$ be the connected component of $B(z, 4\Lambda_K R)$ containing $z$.  Then $B(z, 2R)$ is compactly contained in $V$.  As $X$ is locally connected, $V$ is open \cite[Theorem 25.3]{Munkres}.  Thus by Proposition \ref{LLLC implication}, $V$ is relatively $2\Lambda_K$-$LLC$.  

As $X$ is connected and locally connected, if $S$ is a subset of $X$ which is not all of $X$, the closure of the components of $S$ must intersect $X \bslash S$.  Thus if there is more than one component of $X \bslash B(z, R/2\Lambda_K)$ that intersects $X \bslash B(z,R)$, then we may find points $x$ and $y$ in distinct components of $X \bslash B(z, R/2\Lambda_K)$ such that $d(x,z)=R=d(y,z).$ However, $B(z, R)$ is compactly contained in $V$, and so the second relative $LLC$ condition implies that $x$ and $y$ may be connected in $V \bslash B(z, R/2\Lambda_K)$.   This is a contradiction. 

Let $W$ be the unique component of $X \bslash B(z, R/2\Lambda_K)$ which intersects $X\bslash B(z,R)$, and set $U = X \bslash W$.  Then $U$ is an open, non-compact subset of $X$ with connected complement.  We in fact have that $W \subeq X\bslash B(z, R)$, and so $U \subeq B(z, R)$.  Since $X$ is linearly locally contractible, $B(z, R)$ contracts inside $B(z, \Lambda_K R)$.  Thus the homomorphism $i_*\colon \pi_1(U) \to \pi_1(X)$ induced by the inclusion $i \colon U \to X$ is trivial. Since $X$ is proper, $U$ is relatively compact in $X$.  Proposition \ref{complement prop} now shows that $U$ is homeomorphic to the plane. The inclusions \eqref{disk inclusions} hold by construction. 
\end{proof} 

\begin{proof}[Proof of Theorem \ref{existence}]  We recall the set up. The space $(X,d)$ is a proper, linearly locally contractible, and Ahlfors $2$-regular metric space. The set $K$ is a compact subset of $X$ such that if $x \in K$, and $0< r \leq R_K$, then the ball $B(x, r)$ contracts inside $B(x, \Lambda_K r)$, and moreover
$$\frac{r^2}{C_K} \leq \Hdim^2(B(x, r)) \leq C_K r^2.$$
We let $z$ be an interior point of $K$, and set
$$R_0 =\min \{ \max\{R \geq 0 : \ovl{B}(z, R) \subeq K\}, R_K\} > 0.$$ 
If $X = B(z, R_0)$, then $X$ is compact and contractible.  As discussed in the proof of Proposition \ref{disk size}, this is a contradiction.  Thus we may assume that $X \bslash B(z, R_0) \neq \emptyset.$  

Consider the ball $B_0 = B(z, R_0/4)$.  The definition of $R_0$ implies that $\diam(B_0) \leq  R_K/2$, and that the $2\diam(B_0)$-neighborhood of $B_0$ is contained in $K$.  This easily implies that $B_0$ is relatively Ahlfors $2$-regular and relatively linearly locally contractible with constants $C_K$ and $\Lambda_K$ respectively. In addition, Proposition \ref{loc regular implies loc doubling} implies that $B_0$ has relative Assouad dimension $2$ with constant $64C_K^2$.  Similarly, Proposition \ref{LLLC implication} shows that $B_0$ satisfies the first relative $LLC$ condition with constant $\Lambda_K$.  

If $x,y \in B(z, R_0/16) \subeq B_0$, then $2d(x,y)\leq \diam(B_0)$.  Since $x,y \in B(x, 2d(x,y))$, the first relative $LLC$ condition shows that there is a continuum connecting $x$ to $y$ of diameter no greater than $4\Lambda_Kd(x,y)$. Thus $B(z, R_0/16)$ is of $4\Lambda_K$-bounded turning in $X$.  Furthermore, as a subset of $B_0$, the ball $B(z, R_0/16)$ has relative Assouad dimension at most $2$ with constant depending only on $C_K$.  By Proposition \ref{cqa existence}, there are constants $M, N \geq 1$ and $0<c\leq1$, depending only on $C_K$ and $\Lambda_K$, such that for each pair of points $x, y \in B(z, cR_0/16)$ and each $0 < \ep < d(x,y)$, there is an $(\ep, M)$-quasiarc connecting $x$ to $y$ inside $B(x, Nd(x,y)).$  


By the definition of $R_0$, we have
$$\frac{cR_0}{16\Lambda_K} \leq \frac{R_K}{8\Lambda_K},$$
and that $B(z, cR_0/4) \subeq K.$  Thus by Proposition \ref{disk size}, there is a neighborhood $U$ of $z$ homeomorphic to $\reals^2$ such that 
$$B\left(z, \frac{cR_0}{32\Lambda_K^2}\right) \subeq U \subeq B\left(z, \frac{cR_0}{16\Lambda_K}\right) \subeq B_0.$$ 
Since $U \subeq B_0$, $U$ is also relatively Ahlfors $2$-regular and relatively linearly locally contractible with constants $C_K$ and $\Lambda_K$ respectively. Furthermore, if $x, y \in U$, and $0< \ep < d(x,y)$, there is an $(\ep, M)$-quasiarc connecting $x$ to $y$ inside $B(x, Nd(x,y)).$  Theorem \ref{quasiconvexity main work} now implies that there is a constant $L \geq 1$, depending only on $C_K$ and $\Lambda_K$, such that each pair of points 
$$x, y \in B\left( z, \frac{cR_0}{128\Lambda_K^2 N} \right)$$
may be connected by a path of length no more than $Ld(x,y)$.  

As a subset of $B_0$, the set $U$ has relative Assouad dimension at most $2$ with constant depending only on $C_K$.  Theorem \ref{quasicircle} shows that there are constants $\lambda, C_1, C_2 \geq 1$ depending only on $C_K$ and $\Lambda_K$ such that if $R \leq  R_0/C_1$, then there is a $\lambda$-chord-arc loop $\gamma$ in $U$ such that $\ind(\gamma, z) \neq 0$ and 
\begin{equation}\label{gamma estimates} \frac{R}{C_2} \leq \dist(z, \im(\gamma)) \leq C_2R, \mand \frac{R}{C_2} \leq \diam(\im{\gamma}) \leq C_2R.\end{equation}

Let $A_1 = C_1$, and fix $R \leq R_0/A_1$.  Let $\gamma$ be as described above, and set $\Omega$ to be the inside of $\im{\gamma}$. Then $\ovl{\Omega} = \Omega \cup \im(\gamma)$.  We first claim that 
\begin{equation}\label{omega diam} \ovl{\Omega} \subeq B(z, C_2(4\Lambda_K+2)R).\end{equation}
The proof is similar to Case 2 of Theorem \ref{porosity thm} and Lemma \ref{dist diam}, but we include it for completeness.  Suppose that $x \in \ovl{\Omega}$ but $d(z,x) \geq C_2(4\Lambda_K+2)R.$  Then by \eqref{gamma estimates}
\begin{equation}\label{omega diam 1} \dist(x,\gamma) \geq d(z, x) - \dist(z, \gamma) - \diam(\gamma)  \geq 4C_2\Lambda_K R. \end{equation}
This shows that $x \in \Omega$, and so $\ind(\gamma, x) \neq 0$.  However, since $\ovl{\Omega} \subeq U$, and $U$ is  relatively $\Lambda_K$-linearly locally contractible, $\gamma$ is homotopic to a point with homotopy tracks inside the $2\Lambda_K\diam(\gamma)$ neighborhood of itself.  By \eqref{gamma estimates} and \eqref{omega diam 1}, these tracks do not hit $x$.  This is a contradiction, proving \eqref{omega diam}.  Note that this implies that 
\begin{equation}\label{omega gamma diam} 2\diam(\ovl{\Omega}) \leq \diam(U).\end{equation}
We now claim that 
\begin{equation}\label{lower omega inclusion} 
B\left(z, \frac{R}{2\Lambda_KC_2}\right) \subeq \Omega.
\end{equation}
Let $C(z)$ be the component of $B(z, R/C_2)$ containing $z$.  Since $z \in \Omega$,  \eqref{gamma estimates} implies that $C(z) \subeq \Omega.$  Since $U$ is relatively $2\Lambda_K$-$LLC$, $B(z, R/2\Lambda_KC_2) \subeq C(z)$, establishing the claim.  From \eqref{lower omega inclusion} and \eqref{omega diam}, we see that $\Omega$ satisfies condition $(i)$ of the theorem, with $A_2 = C_2(4\Lambda_K+2).$  

It remains to show that $\Omega$ is quasisymmetrically equivalent to the disk with controlled distortion function.  To do so, we will show that $\ovl{\Omega}$ is $LLC$ and Ahlfors $2$-regular with constants depending only on $\Lambda_K$ and $C_K$.  The desired result will then follows from Theorem \ref{main}. 

We first address the $LLC$ condition.   We will not use the full strength of the fact that $\gamma$ is a chord-arc loop.  Instead, we will only need the weaker diameter condition given by \eqref{three point}.  Note that as $U$ is an open, connected subset of $B_0$, Proposition \ref{LLLC implication} implies that $U$ is relatively $2\Lambda_K$-$LLC$. 

Let $x \in \ovl{\Omega}$ and $r > 0.$  Suppose that $a, b \in B_{\ovl{\Omega}}(x, r).$  By \eqref{omega gamma diam}, we may assume without loss of generality that $B_{\ovl{\Omega}}(x,r)$ is compactly contained in $U$.  Since $a, b \in U$, there is a path $\alpha\colon [0,1] \to U$ such that $\alpha(0)=a$, $\alpha(1)=b$, and $\im\alpha \subeq B_X(x, 2\Lambda_K r).$  Let 
$$t_a = \min\{ t \in [0,1] : \alpha(t) \in \im\gamma\} \mand t_b = \max\{t \in [0,1] : \alpha(t) \in \im\gamma\}.$$
Since $\gamma$ is a $\lambda$-chord-arc curve, $\gamma_{ab} \subeq \im\gamma$ which connects $\alpha(t_a)$ to $\alpha(t_b)$ and satisfies 
$$\diam(\gamma_{ab}) \leq d(\alpha(t_a), \alpha(t_b)) \leq 4\lambda \Lambda_K r.$$
Then 
$$\alpha([0, t_a]) \cup \gamma_{ab} \cup \alpha([t_b, 1]) \subeq B(x, 2\Lambda_K(4\lambda+1) r)$$
is a continuum connecting $a$ to $b$.  This shows that $\ovl{\Omega}$ is $2\Lambda_K(4\lambda+1)$-$LLC_1$.  

Now suppose that $a,b \in \ovl{\Omega}\bslash B_{\ovl{\Omega}}(x,r).$ As above, there is a path $\alpha \colon [0, 1] \to U$ such that $\alpha(0)=a$, $\alpha(1)=b$, and $\im(\alpha) \subeq U \bslash B_{X}(x, r/2\Lambda_K).$ Define $t_a$ and $t_b$ as above; we may write $\gamma = \gamma_1 \cup \gamma_2$, where for $i=1,2$, $\gamma_i$ is a closed subarc of $\gamma$ with endpoints $\alpha(t_a)$ and $\alpha(t_b)$. Suppose that we may find points $$x_1 \in \gamma_1 \cap B\left(x, \frac{r}{4\Lambda_K(2\lambda+1)}\right)  \mand x_2 \in \gamma_2 \cap B\left(x, \frac{r}{4\Lambda_K(2\lambda+1)}\right).$$
Then $$d(x_1, x_2) \leq \frac{r}{2\Lambda_K(2\lambda+1)},$$ and so the by the quasiarc property of $\gamma$, either $\alpha(t_a)$ or $\alpha(t_b)$ is contained in 
$$\ovl{B}\left(x_1,\frac{\lambda r}{2\Lambda_K(\lambda+1)}\right) \subeq B\left(x, \frac{r}{4\Lambda_K}\right).$$
This contradicts the fact that $\im \alpha \subeq U \bslash B_X(x, r/2\Lambda_K).$  Thus there is some $i \in \{1, 2\}$ such that $\alpha([0, t_a]) \cup \gamma_i \cup \alpha([t_b,1])$ connects $a$ to $b$  in 
$$\ovl{\Omega}\bslash  B\left(x, \frac{r}{4\Lambda_K(2\lambda+1)}\right) .$$
Thus we have shown that $\ovl{\Omega}$ is $\lambda'$-$LLC$ with $\lambda' = 4\Lambda_K(4\lambda+1).$

We now show that $\ovl{\Omega}$ is Ahlfors $2$-regular with constant depending only on $\Lambda_K$ and $C_K$.  Let $x \in \ovl{\Omega}$ and $0 \leq r \leq \diam(\ovl{\Omega}).$  Since $U$ is relatively Ahlfors $2$-regular with constant $C_K$, $\ovl{\Omega}$ is as well, and by \eqref{omega gamma diam}, we see that $r \leq \diam(U).$ It follows from the definition of Hausdorff measure that
$$\Hdim^2_{\ovl{\Omega}}(\ovl{B}_{\ovl{\Omega}}(x,r)) \leq 4\Hdim^2_{X}(\ovl{B}_{\ovl{\Omega}}(x,r)) \leq 4\Hdim^2_{X}(\ovl{B}_{X}(x,r)) \leq 4C_Kr^2.$$
	
To show the lower bound, we consider four cases. 

\bigskip\noindent
\textit{Case 1: $x \in \Omega$ and $r < \dist(x, \im\gamma)$.} Since $\ovl{\Omega}$ is $\lambda'$-$LLC$, the connected component of $x$ in $\ovl{B}_{\ovl{\Omega}}(x, r)$ contains the ball $\ovl{B}_{\ovl{\Omega}}(x, r/2\lambda').$  This implies that  there is some $\ep>0$ such that the $\ep$-neighborhood of $\ovl{B}_{X}(x, r/2\lambda')$ is contained in $\ovl{\Omega}$.  As a result, 
$$\Hdim^2_{\ovl{\Omega}}(\ovl{B}_{\ovl{\Omega}}(x, r)) \geq \Hdim^2_{\ovl{\Omega}}\left(\ovl{B}_{\ovl{\Omega}}\left(x, \frac{r}{2\lambda'}\right)\right) =  \Hdim^2_{X}\left(\ovl{B}_{X}\left(x, \frac{r}{2\lambda'}\right)\right) \geq \frac{r^2}{4C_K {\lambda'}^2}.$$

\bigskip\noindent
\textit{Case 2: $x \in \gamma$ and $0 \leq r \leq \diam(\ovl{\Omega}).$}  By Theorem \ref{porosity thm}, there is a constant $C$, depending only on $C_K$ and $\Lambda_K$, and a point $y \in \Omega$ such that 
$$B_{\ovl{\Omega}}(y, r/C) \subeq B_{\ovl{\Omega}}(x, r) \bslash \im\gamma.$$
It follows from Case 1 that 
$$\Hdim^2_{\ovl{\Omega}}(\ovl{B}_{\ovl{\Omega}}(x, r)) \geq\Hdim^2_{\ovl{\Omega}}\left(\ovl{B}_{\ovl{\Omega}}\left(y, \frac{r}{2C}\right)\right) \geq \frac{r^2}{16C_K (C\lambda')^2}.$$

\bigskip\noindent
\textit{Case 3: $x \in \Omega$ and $\dist(x,\im\gamma) \leq r < 4\dist(x, \im\gamma)$.}  By Case 1, we have 
$$\Hdim^2_{\ovl{\Omega}}(\ovl{B}_{\ovl{\Omega}}(x, r)) \geq \Hdim^2_{\ovl{\Omega}}\left(\ovl{B}_{\ovl{\Omega}}\left(x, \frac{r}{4}\right)\right) \geq \frac{r^2}{64C_K\lambda'^2}.$$

\bigskip\noindent
\textit{Case 4: $x \in \Omega$ and $r \geq 4\dist(x, \im\gamma).$} We may find a point $y \in \gamma$ such that $d(x,y) < r/2.$   Then $\ovl{B}_{\ovl{\Omega}}(y, r/2) \subeq B_{\ovl{\Omega}}(x, r)$, and so by Case 2 we have 
$$\Hdim^2_{\ovl{\Omega}}(\ovl{B}_{\ovl{\Omega}}(x, r)) \geq\Hdim^2_{\ovl{\Omega}}\left(\ovl{B}_{\ovl{\Omega}}\left(y, \frac{r}{2}\right)\right) \geq \frac{r^2}{64C_K (C\lambda')^2}.$$

Thus we have shown that $\ovl{\Omega}$ is Ahlfors $2$-regular with constant $64C_K(C\lambda')^2$.  The desired quasisymmetric homeomorphism is now provided by Theorem \ref{main}. 
\end{proof}

\bibliographystyle{plain}
\bibliography{LocUnif}

\end{document}